\documentclass[12pt]{article}
\usepackage{amsfonts}
\usepackage{amssymb}
\usepackage{amsmath}
\usepackage{amsthm}
\usepackage{euscript}
\usepackage{graphics}
\usepackage{verbatim}
\usepackage{color}

\usepackage[margin=1.1in]{geometry}

\usepackage{mathrsfs}

\usepackage{hyperref}

\newcommand\al{\alpha}
\newcommand\be{\beta}
\newcommand\ga{\gamma}

\newcommand\de{\delta}

\newcommand\z{\zeta}
\newcommand\om{\omega}
\newcommand\Om{\Omega}
\newcommand\la{\lambda}
\newcommand\La{\Lambda}
\newcommand{\eps}{\varepsilon}

\newcommand\alde{{\al,\be,\ga,\de}}

\newcommand\R{\mathbb R}

\newcommand\G{{\mathbb G}}
\newcommand\ZZ{\mathbb Z}

\newcommand{\GT}{{\mathbb{GT}}}
\newcommand{\Y}{\mathbb Y}
\newcommand{\GTe}{\mathbb G}

\newcommand\sgn{\operatorname{sgn}}
\newcommand{\wt}{\operatorname {wt}}
\newcommand\Dim{\operatorname{Dim}}
\newcommand\PP{\operatorname{Prob}}
\newcommand\const{\operatorname{const}}

\newcommand{\Z}{\mathcal Z}

\newcommand{\M}{\mathcal M}

\renewcommand{\L}{\mathfrak L}

\newcommand{\TI}{\vartriangleleft}

\newcommand\one{\mathbf1}

\newcommand{\rr}{\mathbf r}

\newcommand{\F}{\mathcal F}

\newcommand\ti{\widetilde}

\newtheorem{theorem}{Theorem}[section]
\newtheorem{proposition}[theorem]{Proposition}
\newtheorem{lemma}[theorem]{Lemma}
\newtheorem{corollary}[theorem]{Corollary}

\theoremstyle{definition}
\newtheorem{definition}[theorem]{Definition}
\newtheorem{remark}[theorem]{Remark}
\newtheorem{example}[theorem]{Example}

\numberwithin{equation}{section}

\title{A quantization of the harmonic analysis on the infinite--dimensional unitary group}

\author{Vadim Gorin and Grigori Olshanski}

\date{}

\begin{document}

\maketitle

\tableofcontents

\newpage

\begin{abstract}
The present work stemmed from the study of the problem of harmonic analysis on the
infinite-dimensional unitary group $U(\infty)$. That problem consisted in the decomposition of a certain
4-parameter family of unitary representations, which replace the nonexisting two-sided regular
representation (Olshanski, J. Funct. Anal., 2003). The required decomposition is governed by certain
probability measures on an infinite-dimensional space $\Om$, which is a dual object to $U(\infty)$. A way
to describe those measures is to convert them into determinantal point processes on the real line;
it turned out that their correlation kernels are computable in explicit form --- they admit a
closed expression in terms of  the Gauss hypergeometric function ${}_2F_1$ (Borodin and Olshanski,
Ann. Math., 2005).

In the present work we describe a (nonevident) $q$-discretization of the whole construction. This
leads us to a new family of determinantal point processes. We reveal its connection with an exotic finite system of $q$-discrete orthogonal
polynomials -- the so-called pseudo big $q$-Jacobi polynomials. The new point processes live on a
double $q$-lattice and we show that their correlation kernels are expressed through the basic
hypergeometric function ${}_2\phi_1$.

A crucial novel ingredient of our approach is an extended version $\G$ of the Gelfand-Tsetlin graph (the
conventional graph describes the Gelfand-Tsetlin branching rule for irreducible representations of unitary groups). We
find the $q$-boundary of $\G$, thus extending previously known results (Gorin, Adv. Math., 2012).
\end{abstract}

\section{Introduction}

\subsection{Origin of the problem}

By the infinite-dimensional unitary group we mean the inductive limit group $U(\infty):=\varinjlim U(N)$.  This group (as well as any its topological completion) is not locally compact and hence does not admit an invariant measure which is a necessary element for constructing the regular or biregular representation. Nevertheless, there exists a 4-parameter family
of unitary representations of $U(\infty)\times U(\infty)$ which are reasonable analogs of the biregular representations of the pre-limit groups $U(N)\times U(N)$ (Olshanski \cite{Ols-JFA}).  The problem of harmonic analysis for $U(\infty)$ in the formulation of \cite{Ols-JFA} consists in the decomposition of these ``generalized biregular representations'' of $U(\infty)\times U(\infty)$ into a continuous integral of irreduciblle representations.

As explained in \cite{Ols-JFA}, for nonexceptional values of the parameters, the decomposition
problem can be reduced to the description of a certain family $\{\M\}$ of probability measures
which live on an infinite-dimensional space $\Om$.

The space $\Om$, initially defined as a kind of a dual object to the group $U(\infty)$, can be
identified with the boundary of the Gelfand--Tsetlin graph --- a graded graph whose vertices of
level $N=1,2,\dots$ represent the irreducible characters of $U(N)$  and the graph structure
reflects the branching rule of characters.

Each measure $\M$ from our family is approximated by a canonical sequence $\M_1,\M_2,\dots$ of
discrete probability measures, which are defined on the growing levels of the Gelfand--Tsetlin
graph. We call  the latter measures the \textit{zw-measures}, and the limit measures $\M$ on $\Om$
are called the \textit{boundary zw-measures}.

In  Borodin-Olshanski \cite{BO-AnnMath} it is shown that every boundary  zw-measure $\M$ can be
turned into (the law of) a determinantal point process on the real line with two punctures, and the
correlation kernel of that process can be explicitly computed: it is expressed  in terms of the
Gauss hypergeometric function ${}_2 F_1$.  At present, this is the only way to describe the
boundary zw-measures $\M$.

The zw-measures $\M_N$ on the levels of the Gelfand--Tsetlin graph play a fundamental role in the
whole theory, because the only existing way to handle the boundary measures $\M$ relies on the
approximation $\M_N\to\M$.

There is a remarkable similarity between the zw-measures and $\be=2$  log-gas-type particle systems
(Forrester \cite{For}) arising in random matrix theory.  Here $\be$ is Dyson's parameter, which is
similar to the Jack parameter in the theory of symmetric functions (Macdonald \cite{Mac}). Like
log-gas systems, the zw-measures admit a deformation corresponding to arbitrary positive values of
$\be$  (Olshanski \cite{Ols-FAA}).  Is it possible to go further and construct other deformations
of the zw-measures, parallel to Macdonald's $(q,t)$-deformation of the canonical scalar product in
the algebra of symmetric functions? In the present work, we get an affirmative answer for  the
simplest non-Jack case, namely, for $q=t$.

At first we believed that the desired $q$-analog of the zw-measures can be constructed with the use
of the $q$-boundary of the Gelfand-Tselin graph --- a concept, which was introduced in Gorin
\cite{G-AM} and then successfully exploited in Borodin-Gorin \cite{BG}. And indeed, we were able to
do it  for certain special values of the parameters corresponding to  \emph{degenerate} versions of
the measures. However,  in the case of general parameters, all our attempts failed. Eventually we
realized what was the cause of failure.  It turned out that, for our purposes, the very notion of
the Gelfand-Tsetlin graph is not suitable and has to be extended.

The key idea of the definition of the \textit{extended Gelfand--Tsetlin graph} is the following:
the vertices of the ordinary Gelfand-Tsetlin graph can be identified with the finite point
configurations on the one-dimensional lattice $\ZZ\subset\R$, which label irreducible
representations of unitary groups. In the extended graph, the lattice $\ZZ$ has to be replaced by
the \textit{double lattice} $\ZZ\sqcup\ZZ$, which we identify with the double $q$-lattice of the
form
$$
\L:=\{\z_- q^m: m\in\ZZ\}\sqcup \{\z_+ q^n: n\in\ZZ\}\subset \R\setminus\{0\},
$$
where $\z_-<0$ and $\z_+>0$ are fixed parameters.  The specific choice of parameters $\z_-$, $\z_+$ does not play a substantial role, the most important is  that they are of opposite sign, so that the lattice lies on both sides of the zero in $\R$.

The lattice $\L$ comes from the Askey scheme of basic hypergeometric orthogonal polynomials (see Koekoek-Swarttouw \cite{KS}).  As explained in \cite{BO-AnnMath}, the pre-limit zw-measures are closely related to certain unnamed \textit{finite} systems of orthogonal polynomials on $\ZZ$;  in \cite{BO-AnnMath}, we called  them the  \textit{Askey-Lesky polynomials}.  It turns out that suitable $q$-analogs of the Askey-Lesky polynomials are the so-called \textit{pseudo big $q$-Jacobi polynomials} (Koornwinder \cite{Koo-Addendum}). These polynomials are eigenfunctions of a second order $q$-difference operator $D$ (a natural $q$-analog of the difference operator associated with the Askey-Lesky polynomials) and they are orthogonal on $\L$.

The idea to use the double lattice $\L$ came to us when we became aware of the papers of Groenevelt and Koelink  \cite{Gro1}, \cite{Gro2}, \cite{GroKoel}. Thanks to them we realized that $D$  admits a suitable selfadjoint version only when it is considered on a double lattice.

\subsection{Main results}

Very briefly, the main results of the present paper are the following:

\medskip

1. \textit{The description of the $q$-boundary of the extended Gelfand-Tsetlin
graph.} We show that the  $q$-boundary can be identified with the space of two-sided
bounded countable point configurations on the lattice $\L$. In the case of the ordinary Gelfand--Tsetlin graph, various approaches to the description of the ($q$-) boundary were devised in Okounkov--Olshanski \cite{OO}, Borodin--Olshanski \cite{BO-AM}, Gorin \cite{G-AM}, Petrov \cite{Petrov-MMJ}.  Our approach is different; it combines
quantative estimates with some  ideas from Gorin--Panova \cite{GP}.

\medskip

2. \textit{The construction of  a $q$-analog of the boundary zw-measures. } These are certain
probability measures on the $q$-boundary of the extended Gelfand--Tsetlin graph; they are obtained from the $N$-particle
orthogonal polynomial ensembles corresponding to pseudo big $q$-Jacobi polynomials.
The key point is the \textit{coherence property}:  the ensembles with different
values of $N=1,2,\dots$ are linked to each other by means of certain canonical stochastic matrices associated with the graph.  The coherency property is a nontrivial hypergeometric identitity; various combinatorial proofs of a similar identity in the case of the ordinary Gelfand--Tsetlin graph were suggested in Olshanski \cite{Ols-FAA}, but our approach is different: we show that our $q$-identity can be derived from the backward shift relation satisfied by the big $q$-Jacobi polynomials.

\medskip

3. \textit{The computation of the correlation functions.} We prove that the measures in item 2 above are determinantal measures and so are completely determined by their
correlation kernels. Those  have the form
\begin{equation}
\label{eq_integrable_kernel} K(x,y)=\const
\frac{\F_0(x)\F_1(y)-\F_1(x)\F_0(y)}{x-y}, \qquad x,y\in\L,
\end{equation}
where $\F_0(x)$ and $\F_1(y)$ are certain functions on $\L$  expressed through the
basic hypergeometric function ${}_2\phi_1$.
Kernels of such a form as in \eqref{eq_integrable_kernel}, are called
\textit{integrable} (in the sense of Its-Izergin-Korepin-Slavnov, see Deift's survey
paper \cite{Deift}). There are a lot of examples of integrable kernels coming from
different models of random matrix theory and other sources, but the above kernel
seems to be the first one which is expressed through basic hypergeometric functions.

\subsection{Organization of the paper}
In Section \ref{sect2}, we introduce the extended Gelfand--Tsetlin graph and the associated  canonical stochastic matrices linking its levels; we call these matrices the \textit{$q$-links}.

Section \ref{sect3} is devoted to the $q$-boundary. We start with generalities concerning the
notion of boundary that we need: it is defined as the set of extreme points of a projective limit
of simplices. Then we pass to our concrete situation where the simplices in question are spanned by
the vertices from the levels of the extended Gelfand--Tsetlin graph and the maps between the
simplices are given by the $q$-links. The description of the boundary is given in Theorems
\ref{thm1}, \ref{thm2}, \ref{thm3}, \ref{thm4}.  Together they constitute our first main result. At
the end of the section we define the correlation functions for arbitrary probability measures on
the boundary and show that they can be obtained, in principle, by a large-$N$ limit transition.
This abstract result is then used in the computation of Section \ref{sect5}.

Section \ref{sect4} deals with the \textit{pre-limit $q$-zw-measures} --- a $q$-analog of the
pre-limit zw-measures. We start with their definition. Next we explain their connection with the
pseudo big $q$-Jacobi polynomials and collect a number of formulas that we need (here our basic
source is recent Koornwinder's paper \cite{Koo-Addendum}, which gives  a further reference to
Groenevelt--Koelink \cite{GroKoel}). Finally, we establish the  coherency property of the pre-limit
$q$-zw-measures (Theorem   \ref{theorem_coherency}).  It implies the existence of the boundary
$q$-zw-measures, which is our second main result.

In the final Section \ref{sect5} we compute the correlation kernels of the boundary $q$-zw-measures (Theorem \ref{theorem_ergodic}). This is the third main result.  The computation is a bit tedious because we have to manipulate with long formulas, but it is quite elementary: the desired kernel is obtained by a direct limit transition in the $N$-th Christoffel--Darboux kernel for the pseudo big $q$-Jacobi polynomials as $N\to\infty$.

\subsection{Acknowledgment}

We thank Yuri Neretin for bringing our attention to papers by Wolter Groenevelt and Erik Koelink;
their ideas helped us very much. We are also grateful to Erik Koelink for valuable comments.

V.~G.\ was partially supported by the NSF grant DMS-1407562.

\section{The extended Gelfand-Tsetlin graph and q-links}\label{sect2}

\subsection{The extended Gelfand-Tsetlin graph $\G$}\label{sect2.1}

Here we introduce a novel object ---- the extended Gelfand--Tsetlin graph and explain how it is
related to the conventional Gelfand--Tsetlin graph.

We recall the definition of the double $q$-lattice $\L\subset\R$:
$$
\L:=\{\z_- q^m: m\in\ZZ\}\sqcup \{\z_+ q^n: n\in\ZZ\}\subset \R\setminus\{0\},
$$
where $q\in(0,1)$, $\z_-<0$, and $\z_+>0$ are parameters.  Unless otherwise stated, these three parameters are assumed to be fixed.

By a \textit{configuration} on $\L$ we mean a subset $X\subset\L$. Elements of $X$ are called \textit{particles}. If $X$ is finite, we enumerate its particles in the increasing order and write $X=(x_1<\dots<x_N)$. The set of $N$-particle configurations is denoted by $\G_N$, $N=1,2,\dots$\,.

 We need a special notion of the \textit{interval} $I(a,b)$ between two points $a<b$ of $\L$:
 $$
I(a,b) := \begin{cases} [a,b)\cap\L, & a<b<0,\\
[a,b]\cap\L, & a<0<b,\\
(a,b]\cup\L, & 0<a<b
\end{cases}
$$
Note that if $a<0<b$, then the interval $I(a,b)$ contains  infinitely many points.

We say that two configurations $X\in\G_{N+1}$ and $Y\in\G_N$ \emph{interlace}
if
$$
y_i\in I(x_i, x_{i+1}), \quad i=1,\dots,N.
$$
Then we write $Y\prec X$ or $X\succ Y$.

Note that the interval $I(x_1,x_{N+1})$ is split into the disjoint union of $N$ subintervals
$I(x_i,x_{i+1})$, $i=1,\dots,N$. The interlacement condition $Y\prec X$ means that each of these
subintervals  contains precisely one particle from $Y$, see Figure \ref{Fig_Interlacing_1} for an
example.

\begin{figure}[h]
\begin{center}
  {\scalebox{1.2}{\includegraphics{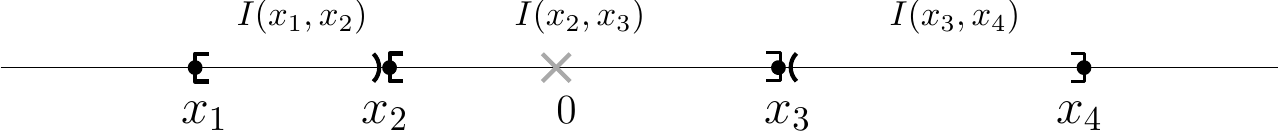}}}
 \end{center}
 \caption{Example of partition into intervals $I(x_i,x_{i+1})$. If $Y\prec X$, then there is precisely one $y$--particle in each interval.
 \label{Fig_Interlacing_1}
  }
\end{figure}

Equivalently, the interlacement relation can be described in the following way.  Let us modify the
above definition of interval by setting
 $$
\ti I(a,b) := \begin{cases} (a,b]\cap\L, & a<b<0,\\
(a,b)\cap\L, & a<0<b,\\
[a,b)\cup\L, & 0<a<b,
\end{cases}
$$
where we also allow $a=-\infty$ and $b=+\infty$. Given $Y=(y_1<\dots<y_N)$ we split the whole
lattice $\L$ into $N+1$ intervals $\ti I(y_i,y_{i+1})$, $i=0,\dots,N$, where $y_0:=-\infty$ and
$y_{N+1}:=+\infty$. In this notation, $X\succ Y$ means that the each such interval  contains
exactly one particle from $X$, see Figure \ref{Fig_Interlacing_2} for an example.

\begin{figure}[h]
\begin{center}
  {\scalebox{1.2}{\includegraphics{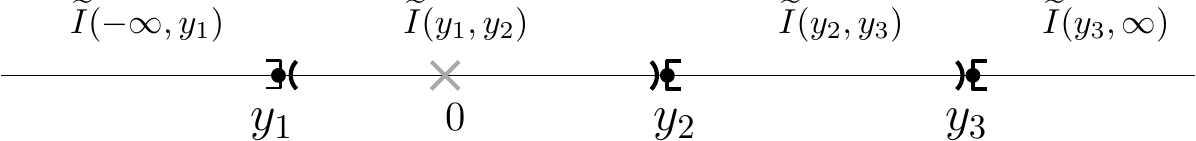}}}
 \end{center}
 \caption{Example of partition into intervals
 $\widetilde I(y_i,y_{i+1})$. If $Y\prec X$, then there is precisely one $x$--particle in each interval.
 \label{Fig_Interlacing_2}}
\end{figure}

\begin{definition}
The  \textit{extended Gelfand-Tsetlin graph} $\G$ is the graded graph whose vertex set is the
disjoint union $\G_1\sqcup\G_2\sqcup\dots$ and the edges are formed by the pairs $X\succ Y$ of
interlacing configurations. We call $\G_N$ the \textit{$N$th level} of the graph.
\end{definition}

As an abstract graph, $\G$ does not depend on the three parameters $q,\z_-,\z_+$, but below we
equip $\G$ with a supplementary structure which is defined with the use of  these parameters.

Recall that the ordinary \textit{Gelfand-Tsetlin graph} (see e.g.\ \cite{BO-AM},\cite{G-AM} and
references therein), is the graded graph $\GT$ whose $N$th level $\GT_N$ ($N=1,2,\dots$) consists
of \textit{signatures} of length $N$: these are $N$-dimensional vectors
$\nu=(\nu_1\ge\dots\ge\nu_N)\in\ZZ^N$. The edges of $\GT$ are formed by pairs $\la\succ\nu$ of
interlaced signatures, where $\la\succ\nu$ means
$$
\la_i\ge\nu_i\ge\la_{i+1}, \qquad i=1,\dots,N, \quad \la\in\GT_{N+1}, \quad \nu\in\GT_N.
$$

There is a natural isomorphism between $\GT$ and the subgraph $\G^+\subset\G$ whose vertices are the configurations entirely contained in the positive part $\L^+$ of the lattice $\L$; here we use the notation
$$
\L^\pm=\{\z_\pm q^n: n\in\ZZ\}\subset\L.
$$
The map $\GT\to\G^+$ is given by
$$
\GT_N\ni\nu\mapsto X\in\G^+, \qquad x_i:=\z_+ q^{\nu_i+N-i}, \quad i=1,\dots,N.
$$
This justifies the name ``extended Gelfand-Tsetlin graph'' given to $\G$.

\subsection{The $q$-links $\La^N_K$}

Here we introduce the $q$-links --- stochastic matrices linking the levels of the graph $\G$. These
matrices arise  when we equip the edges of the graph $\G$ with natural weights depending on
parameter $q$.

\begin{definition}
By the \textit{cotransition probabilities} in $\G$ we mean  the infinite sequence of matrices
$\La^2_1, \La^3_2,\dots$, where $\La^{N+1}_N$ is the matrix of format $\G_{N+1}\times\G_N$ with the
entries
$$
\La^{N+1}_N(X,Y)=\begin{cases}\prod_{i=1}^N |y_i| \cdot (1-q)\dots(1-q^{N})\cdot\dfrac{\prod_{1\le i<j\le N}(y_j-y_i)}
{\prod_{1\le i<j\le N+1}(x_j- x_i)}, & Y\prec X,\\ 0, &\text{otherwise}.\end{cases}
$$
These matrices are called \textit{$q$-links.}
\end{definition}

\begin{remark}
On  the subgraph $\G^+$, the $q$-links coincide with the matrices defined in Gorin \cite{G-AM}. The
$q\to 1$ limit of these matrices describes the branching rule for normalized irreducible characters of the groups $U(N+1)$ under restriction onto $U(N)\subset U(N+1)$, cf.\
\cite{BO-AM, G-AM, Ols-JFA}.
\end{remark}

\begin{proposition} \label{Prop_links_are_stochastic}
The $q$-links are stochastic matrices, i.e.\ for each $N=1,2,\dots$, all the matrix
elements of $\La^{N+1}_N$ are non-negative and for every $X\in\G_{N+1}$
$$
 \sum_{Y\in \G_N} \La^{N+1}_N(X,Y)=1.
$$
\end{proposition}

For the proof we need a preparation.
We endow every edge $X\succ Y$ with a \textit{weight} defined by
$$
\wt(X,Y):=|Y|:=\prod_{y\in Y}|y|
$$
and we  extend this definition to any pair of vertices from adjacent levels by setting
$\wt(X,Y)=0$ if $X$ and $Y$ do not interlace.

Next, for any pair $X\in\G_N$, $Y\in\G_K$, where $N>K$, we set
\begin{gather*}
\Dim(X,Y):=\sum_{X^{(N-1)},\dots,X^{(K+1)}}\wt(X,X^{(N-1)})\wt(X^{(N-1)},X^{(N-2)})\dots
\wt(X^{(K+1)},Y)\\
=|Y|\sum_{X\succ X^{(N-1)}\succ\dots\succ X^{(K+1)}\succ
Y}|X^{(N-1)}|\dots|X^{(K+1)}|
\end{gather*}
(if $N-K=1$, then $\Dim(X,Y):=\wt(X,Y)$).
We call $\Dim(X,Y)$ the \emph{skew dimension}.

The \emph{dimension}\footnote{For  $X\subset\L^+$, the limit quantity
$\lim_{q\to 1} \Dim X$ coincides with the dimension of the corresponding irreducible representation
of the unitary group $U(N)$.} of $X\in\G_N$ is defined in a similar way:
\begin{gather*}
\Dim X:=\sum_{X^{(N-1)},\dots,X^{(1)}}\wt(X,X^{(N-1)})\wt(X^{(N-1)},X^{(N-2)})\dots
\wt(X^{(2)},X^{(1)})|X^{(1)}|\\
=\sum_{X\succ X^{(N-1)}\succ\dots\succ
X^{(1)}}|X^{(N-1)}|\dots|X^{(1)}|=\sum_{X^{(1)}\in\G_1}\Dim(X,X^{(1)}).
\end{gather*}

By convention, the dimension of every vertex of $\G_1$ is equal to $1$.

\begin{lemma} \label{lemma_summation_geometric}
For every $a,b\in\L$ such that $a<b$ one has
$$
b^n-a^n=(1-q^n)\sum_{c\in I(a,b)}|c|\cdot c^{n-1}, \qquad n=1,2,\dots\,.
$$
\end{lemma}

\begin{proof}
We examine separately three cases: $0<a<b$, $a<b<0$, and $a<0<b$, and in all cases the claim is
reduced to summation of geometric progressions.
\end{proof}
In our notations below we use the $q$--Pochhammer symbol:
$$
 (a;q)_n=(1-a)(1-a\cdot q)\cdots (1-a \cdot q^{n-1}).
$$
\begin{lemma} \label{Lemma_dimension_eval}
The quantity $\Dim X$ is finite for any $N$ and any $X\in\G_N$, and it is given by
\begin{equation}
\label{eq_dimension_formula} \Dim X=\frac{\prod_{1\le i<j\le
N}(x_j-x_i)}{(q;q)_1\dots(q;q)_{N-1}}.
\end{equation}
\end{lemma}

\begin{proof}
For $N=1$ the claim holds true:  both sides of the formula equal 1.  Then we use
induction by $N$. Let us temporarily denote the right-hand side of
\eqref{eq_dimension_formula} by $\Dim' X$. By the very definition of the dimension
function it suffices to check the recurrence relation
$$
\Dim' X=\sum_{Y:\, Y\prec X} |Y|\Dim' Y.
$$
Using the Vandermonde determinant evaluation, it reduces to the relation
\begin{equation}
\label{eq_dim_recurrence} \det[x_j^{i-1}]_{i,j=1}^N=(q;q)_{N-1}\sum_{Y:\, Y\prec
X}|Y|\det[y_j^{i-1}]_{i,j=1}^{N-1}, \qquad X\in\G_N, \quad N\ge2.
\end{equation}

We transform the matrix in the left--hand side of \eqref{eq_dim_recurrence} by
subtracting from the $j+1$-th column the $j$-th one, starting with $j=N-1$ and
ending with $j=1$. The first row of the resulting matrix is $(1,0,\dots,0)$ and the
the determinant takes the form
$$
\det[x_{j+1}^i-x_j^i]_{i,j=1}^{N-1}.
$$
 At this stage application of Lemma \ref{lemma_summation_geometric}
completes the proof.
\end{proof}

\begin{proof}[Proof of Proposition \ref{Prop_links_are_stochastic}]
Comparing the definition of the $q$-link $\La^{N+1}_N$ with the result of Lemma
\ref{Lemma_dimension_eval} we see that
$$
\La^{N+1}_N(X,Y)=\frac{\wt(X,Y)\Dim Y}{\Dim X}.
$$
On the other hand, by the very definition of the dimension function, the matrix on the right is
stochastic.
\end{proof}

For $N>K$ we denote by $\La^N_K$ the matrix of format $\G_N\times\G_K$ with the entries
$$
\La^N_K(X,Y)=\frac{\Dim(X,Y) \Dim Y}{\Dim X}.
$$
{}From the definition of the dimension it is clear that $\La^N_K$ is a stochastic matrix and one has
$$
\La^N_K=\La^N_{N-1}\dots \La^{K+1}_K.
$$

 \subsection{Characterization of $q$-links}

In this section we explain how $q$--links are related to Schur polynomials. We start
with recalling a few basic formulas.

Let $\Y$ denote the set of all partitions, which we identify with the Young diagrams. The
\textit{length} of a partition is the number of its nonzero parts (equivalently, the number of
nonzero rows of the corresponding Young diagram). We denote by $\Y(N)$ the subset of partitions
with length at most $N$.  It can be identified with a subset of $\GT_N$ (the set of signatures of
length $N$) consisting of signatures with non-negative coordinates. Clearly, $\Y(1)\subset
\Y(2)\subset\dots$.

We denote by $S_{\la|N}$  the  $N$-variate \textit{Schur polynomial} indexed by a
signature $\la\in\GT_N$. It
is a symmetric Laurent polynomial given by
$$
S_{\la|N}(u_1,\dots,u_N)=\frac{\det[u_j^{\la_i+N-i}]_{i,j=1}^N}{\prod_{1\le i<j\le N}(u_i-u_j)} .
$$
 If $\la\in\Y(N)\subset\GT_N$, then $S_{\la|N}$ is an ordinary polynomial.

The following formula is known as the \textit{$q$-specialization formula} for $S_{\la|N}$:
\begin{equation}
\label{eq_q_spec} S_{\la|N}(1,\dots,q^{N-1})=\prod_{1\le i<j\le
N}\frac{q^{\la_j+N-j}-q^{\la_i+N-i}}{q^{N-j}-q^{N-i}}
\end{equation}
(see e.g. Macdonald \cite[Ch. I, Section 3, Example 1]{Mac}). It is immediately obtained from the Vandermonde determinant evaluation.

Let us set
$$
\eps_N=(N-1,N-2,\dots,0), \qquad q^{\la+\eps_N}=(q^{\la_1+N-1}, q^{\la_2+N-2},\dots q^{\la_N})\in\R^N.
$$

 The following \textit{symmetry relation} is easily obtained from the definition of the Schur polynomials:
\begin{equation}\label{eq-symmetry}
 \frac{S_{\la|N}(q^{\nu+\eps_N})}{S_{\la|N}(q^{\eps_N})}= \frac{S_{\nu|N}(q^{\la+\eps_N})}{S_{\nu|N}(q^{\eps_N})}, \qquad \la,\nu\in\GT_N.
\end{equation}

 Given $\nu\in\Y(N)$ and $X=(x_1<\dots<x_N)\in\G_N$, we  set $ S_{\nu|N}(X)=S_{\nu|N}(x_1,\dots,x_N)$ and
\begin{equation}\label{eq-S}
\ti
S_{\nu|N}(X):=\frac{S_{\nu|N}(X)}{S_{\nu|N}(1,\dots,q^{N-1})}=\frac{S_{\nu|N}(X)}{S_{\nu|N}(q^{\eps_N})}.
\end{equation}

The next proposition provides a characterization of the $q$-links by a system of
linear equations on their entries.

\begin{proposition}\label{prop-links}
Let $N>K$ and $X\in\G_N$ be fixed.

{\rm(i)} The row entries $\La^N_K(X,Y)$, where $Y$ ranges over $\G_K$, satisfy the following system of relations
\begin{equation}\label{eq-branching1}
\sum_{Y\in\G_K}\La^N_K(X,Y) \ti S_{\nu|K}(Y)=\ti S_{\nu|N}(X), \qquad \nu\in\Y(K).
\end{equation}

{\rm(ii)} Conversely, let  $M_K$ be a probability measure on $\G_K$ such that $M_K(Y)$ vanishes unless $Y$ is contained in a certain bounded interval  $[a,b]\subset\R$. If the quantities $M_K(Y)$ satisfy the system \eqref{eq-branching1}, that is,
\begin{equation}\label{eq-moments}
\sum_{Y\in\G_K}M_K(Y) \ti S_{\nu|K}(Y)=\ti S_{\nu|N}(X), \qquad \nu\in\Y(K),
\end{equation}
then $M_K=\La^N_K(X,\,\cdot\,)$.
\end{proposition}

Observe that for any $K<N$ and $Y\in\GT_K$,  the quantity $\La^N_K(X,Y)$ vanishes
unless $Y$ is contained in the closed interval $[x_1,x_N]\subset\R$. It follows that the probability measure $\La^N_K(X,\,\cdot\,)$ on $\G_K$ satisfies the assumption of item (ii) above.

\begin{proof}
(i) Since $\La^N_K=\La^N_{N-1}\La^{N-1}_K$ for $N-K>2$, it suffices to prove \eqref{eq-branching1}
in the case $K=N-1$. Then it takes the form
$$
\det[x_j^{\nu_i+N-i}]_{i,j=1}^N
=(-1)^{N-1} (q;q)_{N-1}\sum_{Y:\, Y\prec X}\frac{S_{\nu|N}(1,\dots,q^{N-1})}{S_{\nu|N-1}(1,\dots,q^{N-2})}|Y| \det[y_j^{\nu_i+N-1-i}]_{i,j=1}^{N-1},
$$
where we assume $\nu\in\Y(N-1)$, write $X=(x_1<\dots<x_N)$, $Y=(y_1<\dots<y_{N-1})$, and the factor
$(-1)^{N-1}$ on the right arises because the product $\prod_{i<j}(x_i-x_j)$ in  the denominator of
the formula for $S_{\nu|N}$  differs by the factor  $(-1)^{N(N-1)/2}$ from the product
$\prod_{i<j}(x_j-x_i)$ in our formula for $\La^N_{N-1}$.

Using the  formula for the $q$-specialization of the Schur polynomials \eqref{eq_q_spec}
 we may rewrite the desired equality as
$$
\det[x_j^{\nu_i+N-i}]_{i,j=1}^N
=(-1)^{N-1}\prod_{i=1}^{N-1}(1-q^{\nu_i+N-i})\sum_{Y:\, Y\prec X}|Y| \det[y_j^{\nu_i+N-1-i}]_{i,j=1}^{N-1}.
$$

Because $\nu\in\Y(N-1)$, we have $\nu_N=0$, so that  the last row of the matrix in the left-hand
side has the form $(1,\dots,1)$.  Thus, we can argue as in the proof of Lemma
\ref{Lemma_dimension_eval}: we transform the matrix in the same way and then apply Lemma
\ref{lemma_summation_geometric}.

(ii) The set $\G_K$ can be viewed as the subset of $\R^K$ formed by the vectors whose coordinates
belong to $\L$ and strictly increase. Therefore, there is a unique extension of $M_K$  to a
symmetric measure of total mass $K!$ on $\L^K\subset\R^K$, where ``symmetric'' means ``invariant
under permutations of the coordinates in $\R^K$''.  Let us denote the latter measure by $\ti M_K$.
By the assumption, the support of $\ti M_K$ is bounded, so that $\ti M_K$ is uniquely determined by
its moments.

On the other hand, the moments of $\ti M_K$ are uniquely determined by the relations
\eqref{eq-moments}. Indeed, these relations provide an explicit expression for the
integrals of arbitrary  Schur polynomials against $\ti M_K$. Because $\ti M_K$  is
symmetric and the Schur polynomials  $S_{\nu|K}$ span the whole space of symmetric
polynomials in $K$ variables, \eqref{eq-moments}  contains the complete information
about the moments of $\ti M_K$.

This completes the proof.
\end{proof}

We have already mentioned that on the subgraph $\G^+\subset\G$, our $q$-links coincide with the
$q$-links for the ordinary Gelfand-Tsetlin graph that were introduced in Gorin \cite{G-AM}. That
is, assuming for simplicity $\z_+=1$, we have
$$
\La^N_K(q^{\la+\eps_N}, q^{\mu+\eps_K})=\ti\La^N_K(\la,\mu), \qquad \la\in\GT_N, \quad \mu\in\GT_K,
$$
where $\ti\La^N_K$ is our notation for the $q$-links from \cite{G-AM}.

As explained in \cite{G-AM}, the links $\ti\La^N_K$ are characterized by the following system of
relations: For fixed $N>K$ and $\la\in\GT_N$,  one has
\begin{equation}\label{eq-branching2}
\sum_{\mu\in\GT_K}\ti\La^N_K(\la,\mu)\frac{S_{\mu|K}(u_1,\dots,u_K)}{S_{\mu|K}(1,q^{-1},\dots,q^{1-K})}=\frac{S_{\la|N}(u_1,\dots,u_K, q^{-K}, \dots,q^{1-N})}{S_{\la|N}(1,q^{-1},\dots,q^{1-N})}.
\end{equation}

Let us show how \eqref{eq-branching2} is connected to \eqref{eq-branching1}.  We rewrite
\eqref{eq-branching2} in the form
$$
\sum_{\mu\in\GT_K}\ti\La^N_K(\la,\mu)\frac{S_{\mu|K}(u_1 q^{K-1},\dots,u_K q^{K-1})}{S_{\mu|K}(q^{K-1}, \dots,1)}=\frac{S_{\la|N}(u_1 q^{N-1},\dots,u_K q^{N-1}, q^{N-K-1}, \dots,1)}{S_{\la|N}(q^{N-1},\dots,1)}.
$$
Now let us substitute
$$
(u_1,\dots,u_K)=(q^{\nu_1}, \dots,q^{\nu_K-K+1}), \qquad \nu\in\Y(K),
$$
and next apply the symmetry relation \eqref{eq-symmetry}. Then we get \eqref{eq-branching1}.

\begin{remark}
The relations \eqref{eq-branching2} say that the links $\ti\La^N_K$ come from the branching rule for appropriately normalized irreducible characters of the unitary groups. This suggests the idea to interpret the functions  $\nu\mapsto\ti S_{\nu|N}(X)$ as a surrogate of characters.
\end{remark}

\section{The $q$-boundary of the extended Gelfand-Tsetlin graph}\label{sect3}

The aim of this section is to identify the \emph{boundary} of the extended Gelfand--Tsetlin graph
equipped with $q$--links. For a historic motivation, we remark that the identification of the
boundary of the conventional Gelfand--Tsetlin graph at $q=1$ is equivalent to the classification of
extreme characters of the group $U(\infty)$ (the extreme characters are in a one-to-one correspondence with the quasi-equivalence classes of finite factor respresentations), cf.\ \cite{BO-AM, Ols-JFA} and references
therein.

\subsection{Preliminaries on  boundaries}

In this short preliminary section we recall general definitions and results related to the notion
of a ``boundary''. The material is more or less standard, but various authors present it in
different ways, see e.g. Dynkin \cite{Dynkin}, Diaconis--Freedman \cite{DF},
Kerov--Okounkov--Olshanski \cite{KOO}, Winkler \cite{Winkler}.

Let $\M(\mathfrak{X})$ denote the set of all probability Borel measures on a given Borel (=measurable) space $\mathfrak{X}$.  By a \textit{link} between Borel spaces $\mathfrak{X}$ and $\mathfrak{Y}$ we mean a Markov kernel $\La=\La(x,dy)$.  We will need only the case when the second space $\mathfrak{Y}$ is discrete. Then a link is a function $\La(x,y)$ on $\mathfrak{X}\times \mathfrak{Y}$, with nonnegative values, which is a Borel function in $x\in \mathfrak{X}$, and such that $\sum_{y\in \mathfrak{Y}}\La(x,y)=1$ for any $x\in \mathfrak{X}$. In particular, if $\mathfrak{X}$ is also discrete, this means that  $\La$ is a stochastic matrix of format $\mathfrak{X}\times \mathfrak{Y}$.  We represent a link by a dash arrow.  Every link $\La:  \mathfrak{X}\dasharrow\mathfrak{Y}$ induces a map $\M(\mathfrak X)\to\M(\mathfrak Y)$, which is an affine map of convex sets. We write it as $M\mapsto M\La$, where the measure $M\in\M(\mathfrak{X})$ is viewed as a row vector.

Assume we are given an infinite  chain
$$
\cdots \dasharrow \Om_{N+1} \dasharrow \Om_N\dasharrow\cdots \dasharrow\Om_2\dasharrow\Om_1,
$$
of discrete spaces together with links $\La^{N+1}_N:\Om_{N+1}\dasharrow\Om_N$.  These links give rise to maps $\M(\Om_{N+1})\to \M(\Om_N)$, which allow us to form the projective limit space $\varprojlim\M(\Om_N)$. By the very definition of projective limit,  elements of the space $\varprojlim\M(\Om_N)$  are infinite sequences $\{M_N\in\M(\Om_N): N=1,2,\dots\}$ satisfying the \textit{coherency relation}
 $$
 M_{N+1}\La^{N+1}_N=M_N, \qquad N=1,2,\dots\,.
 $$
Such sequences $\{M_N\}$ are called \textit{coherent systems} (of probability measures).

In what follows we assume that the set $\varprojlim\M(\Om_N)$ is nonempty.
We equip it with the Borel structure inherited from the natural embedding
$$
\varprojlim\M(\Om_N)\to \prod_{N=1}^\infty\M(\Om_N).
$$

The product space on the right has a natural structure of a convex set and $\varprojlim\M(\Om_N)$  is its convex subset.

\begin{definition}
By the \textit{minimal boundary} of such a chain $\{\Om_N, \La^{N+1}_N: N=1,2,\dots\}$ we mean the set $\Om:=\operatorname{Ex}(\varprojlim\M(\Om_N))$ of the extreme points of $\varprojlim\M(\Om_N)$.
Given $\om\in\Om$, we will denote by $M^{(\om)}=\{M^{(\om)}_N\}$ the coherent system represented by $\om$.
\end{definition}

Below we often drop the adjective ``minimal'' and write simply ``boundary''. 

\begin{theorem}\label{thm-boundary}
The set\/ $\Om$ is a Borel subset of the projective limit space\/ $\varprojlim\M(\Om_N)$, so that we may form the space $\M(\Om)$ of probability Borel measures on\/ $\Om$.

For every coherent system $M=\{M_N\}$ there exists a unique measure $\sigma\in\M(\Om)$ such that $M=\int_\Om M^{(\om)}\sigma(d\om)$ in the sense that
$$
M_N(y)=\int_\Om M^{(\om)}_N(y)\sigma(d\om) \qquad \text{\rm for every $N=1,2,\dots$ and every $y\in\Om_N$.}
$$

Conversely, every measure $\sigma\in\M(\Om)$ generates in this way a coherent system, so that we get a bijection  $\M(\Om)\to\varprojlim\M(\Om_N)$.
\end{theorem}

\begin{proof} See Olshanski \cite[Theorem 9.2]{Ols-JFA}.
\end{proof}

We will say that $\sigma$ is the \textit{boundary measure} of the corresponding coherent system $M=\{M_N\}$.

For every $N$ we define a link $\La^\infty_N: \Om\dasharrow \Om_N$ by setting  $\La^\infty_N(\om,y)=M^{(\om)}_N(y)$ for $y\in\Om_N$. Here we use the fact that $\Om$ is a Borel subset, which guarantees that $\La^\infty_N(\om,y)$ is a Borel function in $\om$.
Note that
$$
\La^\infty_{N+1}\La^{N+1}_N=\La^\infty_N, \qquad N=1,2,\dots,
$$
where the product on the left is the natural composition.

\begin{example}[The Pascal triangle]
Here is a simple illustrative example. Set $\Om_N:=\{0,1,\dots,N\}$ and define the links $\La^{N+1}_N$ by
$$
\La^{N+1}_N(m,m-1)=\frac{m}{N+1}, \quad \La^{N+1}_N(m,m)=\frac{N+1-m}{N+1}
$$
with all other entries $\La^{N+1}_N(m,n)$ being equal to $0$. Then the boundary $\Om$ can be identified with the unit segment $[0,1]$  and the links $\La^\infty_N$ have the form
$$
\La^\infty_N(\om,n)=\binom Nn \om^n(1-\om)^{N-n},
$$
These claims follow from the classical de Finetti theorem, cf.\ Feller \cite[ch. VII, Section 4]{Feller}.
\end{example}

For any $N>K$ we define the link $\La^N_K: \Om_N\dasharrow\Om_K$ as the composition
$$
\La^N_K=\La^N_{N-1}\dots\La^{K+1}_K.
$$

\begin{definition}
An infinite sequence $\{x(N)\in\Om_N\}$ is called \textit{regular} if for fixed every $K=1,2,\dots$ there exists a weak limit of probability measures on $\Om_K$,
$$
M_K:=\lim_{N\to\infty}\La^N_K(x(N),\,\cdot\,),
$$
which is a probability measure, too.  (Since $\Om_K$ is discrete, this simply means convergence on every subset of $\Om_K)$.) Two regular sequences are said to be \textit{equivalent} if the corresponding limit measures coincide for all $K$.  The set of equivalence classes of regular sequences is called the \textit{Martin boundary} of the chain $\{\Om_N, \La^{N+1}_N\}$.
\end{definition}

It is readily seen that the limit measures $M_K$, $K=1,2,\dots$ coming from a regular sequence form a coherent system. This allows us to identify the Martin boundary with a set of coherent systems and hence with a subset of $\M(\varprojlim\Om_N)$.

\begin{theorem} \label{Theorem_minimal_in_Martin}
As a subset of $\M(\varprojlim\Om_N)$, the Martin boundary  contains the minimal
boundary $\Om=\operatorname{Ex}(\M(\varprojlim\Om_N))$. In other words,  any point
$\om\in\Om$ can be approximated by a sequence $\{x(N)\in\Om_N\}$ in the sense that
$$
\lim_{N\to\infty}\La^N_K(x(N),y)=\La^\infty_K(\om,y)
$$
for every fixed $K$ and $y\in\Om_K$.
\end{theorem}

\begin{proof}
See Okounkov-Olshanski \cite[Theorem 6.1]{OO}.
\end{proof}

It may happen that the Martin boundary is strictly larger than the minimal boundary. However, in many concrete examples both boundaries coincide.

For instance, in the case of the Pascal triangle, the Martin and minimal boundaries coincide, and the convergence of a sequence $n(N)$  to a boundary point $t\in[0,1]$ means that $\lim_{N\to\infty}n(N)/N=t$.

\subsection{The spaces $\G_\infty$ and $\bar\G_\infty$}\label{sect3.2}

In this section we explain the basic properties of the set $\G_\infty$, which will
later turn out to be the desired boundary of the extended Gelfand--Tsetlin graph.

We say that a configuration $X$ on $\L$ is \textit{bounded} if it is contained in a
bounded closed segment $[a,b]\subset\R$.  The minimal such segment is denoted by
$I(X)$.  Let $\bar\G_\infty$ denote the set of all bounded configurations on $\L$
and $\G_\infty\subset\bar\G_\infty$ be the set of all \textit{infinite}  bounded
configurations. Obviously,
$$
\bar\G_\infty=\G_\infty\sqcup\{\varnothing\}\sqcup\G_1\sqcup\G_2\sqcup\dots\,.
$$

Let us enumerate elements of a configuration $X\in\bar\G_\infty$ in the decreasing
order of their absolute values.  To resolve a possible ambiguity arising when two
elements of opposite sign have the same absolute value (which may happen only if
$\log_q|\z_-|-\log_q\z_+\in\ZZ$) we agree to put the positive element first. We
denote by $(x_{(1)}, x_{(2)}, \dots)$ the resulting sequence and complete it with
infinitely many $0$'s in the case when $X$ finite.

\begin{definition}
We call $(x_{(1)}, x_{(2)}, \dots)$ the \textit{variational series} of $X$.
\end{definition}

We set $\bar\L=\L\cup\{0\}$ and equip this set with the topology induced from the
ambient space $\R$. Let $\bar\L^\infty$ be the infinite product space
$\bar\L\times\bar\L\times\dots$ equipped with the product topology. The
correspondence $X\mapsto (x_{(1)}, x_{(2)}, \dots)$ determines an injective map
$\bar\G_\infty\to\bar\L^\infty$, so that we may regard $\bar\G_\infty$ as a subset
of  $\bar\L^\infty$ and equip it with the induced topology.

\begin{proposition}
In this topology,  $\bar\G_\infty$ is a locally compact space.
\end{proposition}

\begin{proof}
From the definition of the variational series it follows that its terms decay exponentially fast.
More precisely, one has
\begin{equation}\label{eq-bound}
|x_{(2n)}|\le |x_{(1)}|q^{n-1}, \quad |x_{(2n+1)}|\le |x_{(1)}|q^n, \qquad n=1,2,\dots,
\end{equation}
because among the first $m$ terms of the series one can always find at least $\lceil
m/2\rceil$ terms of the same sign. It follows from \eqref{eq-bound}  that every
subset of the form $\{X\in\bar\G_\infty: |x_{(1)}|\le\const\}$ is compact. On the
other hand, such subsets are open and, as the constant factor gets large, they
exhaust the whole space. This completes the proof.
\end{proof}

The next proposition shows that the topology on  $\bar\G_\infty$ can be also
described  without recourse to variational series. We use this result below in the
proof of Corollary \ref{corrfunctions}.

\begin{proposition}\label{prop-topology}
A fundamental system of neighborhoods of a given configuration $X\in\bar\G_\infty$ is formed by the
``$\eps$-neighborhoods'' $V_\eps(X)$,  $\eps>0$, where  $V_\eps(X)$ consists of all
configurations that coincide with $X$ outside the open interval $(-\eps,\eps)$.
\end{proposition}

\begin{proof}
We examine separately two cases: $X\in\G_\infty$ and $X\in\G_N$ for some $N$.

\emph{Step} 1. Assume $X\in\G_\infty$.  For $k=1,2,\dots$ we set
$$
U_k(X)=\{Z\in\bar\G_\infty: z_{(1)}=x_{(1)}, \dots, z_{(k)}=x_{(k)}\}.
$$
By the very definition of the topology in $\bar\G_\infty$, the sets $U_k(X)$ form a fundamental
system of neighborhoods of $X$ (here we use the fact that $x_{(n)}\ne0$ for all $n=1,2,\dots$\,.

For every fixed $k$, the set $V_\eps(X)$ is contained in $U_k(X)$ if $\eps$ is so small that all the points $x_{(1)},\dots,x_{(k)}$ lie outside $(-\eps,\eps)$.

On the other hand, every set $V_\eps(X)$ is open. Indeed, let $k$ be the first integer such that $x_{(k+1)}$ lies in the interval $(-\eps,\eps)$. Then all points $x_{(n)}$ with $n\ge k+1$ also lie in that interval. It follows that $V_\eps(X)$ coincides with set
$$
\{Z\in\bar\G_\infty: z_{(1)}=x_{(1)}, \dots, z_{(k)}=x_{(k)}, \quad |z_{(k+1)}|<\eps\},
$$
which is obviously open. This proves the desired claim in the case $X\in\G_\infty$.

\emph{Step} 2. Assume now $X\in\G_N$. Then we modify the above argument as follows. We take $k>N$ and replace the sets $U_k(X)$ by the sets of the form
$$
U_{k,\de}(X)=\{Z\in\bar\G_\infty: z_{(1)}=x_{(1)}, \dots, z_{(N)}=x_{(N)}; \qquad |z_{(n)}|<\de, \quad N+1\le n\le k\},
$$
where $\de>0$.
These sets form a fundamental system of neighborhoods of $X$.

Evidently,  $V_\eps(X)$ is contained  in $U_{k,\de}$ if $\eps$ is so small than $\eps<\de$ and $X$ lies outside $(-\eps,\eps)$.

Finally, we have to prove that $V_\eps(X)$ is open. It actually suffices to do this for small $\eps$ --- so small that $X$ lies outside $(-\eps,\eps)$. Then $V_\eps(X)$ coincides with the open set
$$
\{Z\in\bar\G_\infty: z_{(1)}=x_{(1)}, \dots, z_{(N)}=x_{(N)}, \quad |z_{(N+1)}|<\eps\}.
$$

This completes the proof.
\end{proof}

Note that a sequence $\{X(N)\in\G_N: N=1,2,\dots\}$ converges to a configuration $X\in\G_\infty$ if and only if the variational series of $X(N)$ stabilizes to the variational series of $X$ in the sense that for every fixed $k=1,2,\dots$, there exists $N_0(k)$ such that $x_{(k)}(N)=x_{(k)}$ for all $N\ge N_0(k)$.

Note also that any configuration $X\in\G_\infty$ is a limit of a sequence $\{X(N)\in\G_N:
N=1,2,\dots\}$. For instance, one can take as $X(N)$ the collection of the first $N$ terms of the
variational series of $X$.

We need one more remark. Let $S_\nu$ denote the Schur function indexed by a partition $\nu\in\Y$.
Given $X\in\bar\G_\infty$, we denote by  $S_\nu(X)$ the evaluation of $S_\nu$ at the collection
$\{x:\, x\in X\}$. The definition makes sense because, due to the bound \eqref{eq-bound}, the sums
$\sum_{x\in X}|x|^k$, $k=1,2,\dots$, are finite.  The same bound also shows that  the function
$S_\nu(X)$ is continuous on $\bar\G_\infty$.

\subsection{Formulation of results}

 \label{Section_formulation_boundary}

For the sake of readability, our description of the boundary of the chain $\{\G_N, \La^{N+1}_N\}$ is divided into four claims.

\begin{theorem}\label{thm1}
The Martin boundary of the chain  $\{\G_N, \La^{N+1}_N\}$ can be identified with the set \/
$\G_\infty$ of infinite bounded configurations on the lattice $\L$.

More precisely, a sequence $\{X(N)\in\G_N\}$ is regular if and only if $\{X(N)\}$  stabilizes to a configuration $X\in\G_\infty$, and the correspondence  $\{X(N)\}\mapsto X$
 established in this way determines a bijection of the Martin boundary onto\/ $\G_\infty$.
 \end{theorem}

Recall  that every element of the Martin boundary is represented by a coherent system
$M=(M_1,M_2,\dots)$, and denote by $M^{(X)}=(M^{(X)}_1,M^{(X)}_2,\dots)$  the coherent system
corresponding to a given configuration $X\in\G_\infty$. The next theorem provides a
characterization of the measures $M^{(X)}_K$.

In the next theorem we use the notation
$$
\ti S_\nu(X)=\frac{S_\nu(X)}{S_\nu(1,q,q^2,\dots)},
$$
cf. \eqref{eq-S}.

\begin{theorem}\label{thm2}
Let us fix $X\in\G_\infty$ and  $K=1,2,\dots$, and let $Y$ range over $\G_K$. Then $M^{(X)}_K(Y)$ vanishes unless $I(Y)\subseteq I(X)$ and we have
$$
\sum_{Y\in\G_K, \, I(Y)\subseteq I(X)}M^{(X)}_K{(X)}(Y) \ti S_{\nu|K}(Y)=\ti S_\nu(X) \qquad \text{\rm for any $\nu\in\Y(K)$.}
$$
Moreover, $M^{(X)}_K$ is the only probability distribution on $\G_K$ with these properties.
\end{theorem}

The following theorem is a law of large numbers.

\begin{theorem}\label{thm3}
Let us fix $X\in\G_\infty$. Given $L=1,2,\dots$, we consider $Y\in\G_L$ as a random element with law $M^{(X)}_L$, so that the terms $y_{(1)}(L),  y_{(2)}(L), \dots$ of the variational series of $Y$ become random variables.

For any fixed $k=1,2,\dots$, the probability of the event $y_{(k)}(L)=x_{(k)}$ tends to\/ $1$ as $K\to\infty$.
\end{theorem}

Finally, we identify the minimal boundary.

\begin{theorem}\label{thm4}
The minimal boundary of the chain $\{\G_N, \La^{N+1}_N\}$ coincides with the Martin boundary and
hence can be identified with\/ $\G_\infty$. Moreover, under this identification the Borel structure
of the minimal boundary coincides with the natural Borel structure of the space\/ $\G_\infty$.
\end{theorem}

This implies, in particular, that the link $\La^\infty_K$ has format $\G_\infty\times\G_K$ and, for every $X\in\G_\infty$, the probability distribution $\La^\infty_K(X,\,\cdot\,)$ coincides with $M^{(X)}_K$.

The proofs are given in the next two subsections.

\subsection{Qualitative estimates}

A crucial part of the proof of theorems of Section
\ref{Section_formulation_boundary} is contained in the two qualitative lemmas which
we present next.

\begin{lemma} \label{lemma1}
Fix an arbitrary $k=1,2,\dots$. Assume  $N>k$, let  $X\in\G_N$ be arbitrary, and let $Y\in\G_{N-1}$
be the random configuration distributed according to $\La^N_{N-1}(X,\,\cdot\,)$.  Finally, let
$(x_{(i)})$ and $(y_{(i)})$ be the variational series for $X$ and $Y$, respectively.

 Then for some $c>0$
$$
\PP\big( y_{(1)}=x_{(1)},\dots, y_{(k)}=x_{(k)}\big)\ge 1- \frac {e^{-c N}}{c},
$$
where $c$ might depend on $q$, $\z_\pm$, but not on $X$.
\end{lemma}

\begin{proof}
Write as usual $X=(x_1<\dots<x_N)$, $Y=(y_1<\dots<y_{N-1})$. To make clear the main idea of the
argument we first examine the simplest case when $k=1$.  There are two subcases depending on
whether $x_{(1)}$ is positive or negative, but in these two cases the argument is the same. For
definiteness we will assume that $x_{(1)} >0$. Then $y_{(1)}=x_{(1)}$ means $y_{N-1}=x_N$, so that
we have to prove that the probability of the event $y_{N-1} < x_N$ is exponentially small as $N$
gets large.

Let us fix $y_1,\dots,y_{N-2}$ and consider the corresponding conditional distribution of $y_{N-1}$. Below the symbol $\PP(\,\cdot\,)$ refers to this distribution.

It suffices to prove that $\PP\big(y_{N-1}<x_N\big)$
 is exponentially small with some constants that can be chosen independently of $X$ and $y_1,\dots,y_{N-2}$.

We introduce a shorthand notation:
$$
x:=x_N, \qquad y:=y_{N-1}, \qquad y^*:=y_{N-2}.
$$

We are going to prove the following  two estimates (uniform on $X$ and $y_1,\dots,y_{N-2}$):
\begin{gather}
\PP\big(y=x q\big)\le c_1 e^{-c_2 N}\PP\big(y=x\big),  \label{eq.A} \\
\PP\big(y\le xq\big)\le c_3 \PP\big(y=x q\big)     \label{eq.B}.
\end{gather}
They obviously imply the desired estimate.

Recall that the probability of a given configuration $Y\in\G_{N-1}$ is given by
$$
\PP(Y)=\La^N_{N-1}(X,Y)=\prod_{i=1}^{N-1}|y_i|\cdot\dfrac{(q;q)_{N-1}\prod_{1\le i<j\le N-1}(y_j-y_i)}{\prod_{1\le i<j\le N}(x_j-x_i)}\cdot\one_{Y\prec X}.
$$
It implies
\begin{equation}\label{eq.C}
\frac{\PP\big(y=x q\big)}{\PP\big(y=x\big)}=q\prod_{i=1}^{N-2}\frac{x q-y_i}{x-y_i}.
\end{equation}
Evidently, all the $N-2$ factors in the product on the right are strictly positive and less than 1. We claim that, moreover, they
are uniformly separated from 1. Indeed, for every $i=1,\dots,N-2$,
$$
\frac{xq-y_i}{x-y_i}=\dfrac{q-y_i/x}{1-y_i/x},
$$
which does not exceed $q$ if $y_i$ is positive, and does not exceed $(1+q)/2$ if $y_i$ is negative
(the latter holds because $|y_i|\le x_N$). This implies \eqref{eq.A}.

To verify \eqref{eq.B} we consider separately two cases depending on whether $y^*$
is positive or negative.

Assume $y^*$ is positive, which means $y^*=xq^m$ with some positive integer $m$. Then, in \eqref{eq.B},  $y$ may take the values $xq^n$, where $n=1,2,\dots,m-1$, and we use the trivial bound
\begin{equation}\label{eq.D}
\frac{\PP\big(y=xq^n\big)}{\PP\big(y=xq\big)}=q^{n-1}\prod_{i=1}^{N-2}\frac{x q^{n}-y_i}{x q-y_i}\le q^{n-1}, \qquad n=1,2,\dots,
\end{equation}
which immediately gives us \eqref{eq.B}.

If $y^*$ is negative, then, in \eqref{eq.B},  $y$ takes positive values $xq,
xq^2,\dots$ and negative values $y^*q, y^*q^2,\dots$\,. The probabilities of the
positive values are estimated as above. For the negative values we have
\begin{equation}\label{eq.E}
\frac{\PP\big(y=y^*q^n \big)}{\PP\big(y=x q\big)}=\frac{|y^*|}{x}q^{n-1}
\prod_{i=1}^{N-2}\frac{|y_i|-|y^*|q^n}{|y_i|+x q}\le q^{n-1}
\prod_{i=1}^{N-2}\frac{|y_i|}{|y_i|+x q}.
\end{equation}
Since $|y_i|\le x$ for all $i$, every factor in the last product does not exceed $(1+q)^{-1}$. So the product becomes exponentially small as $N$ gets large, while summation over $n=1,2,\dots$ brings only the constant factor $(1-q)^{-1}$. This completes the proof of \eqref{eq.B}.

Let us proceed now to the case $k>1$. Let the first $k-1$ terms of the variational series for $X$
involve $r$ negative elements and $s$ positive elements, where $r+s=k-1$, and assume for
definiteness that the $k$th term is positive (in the case it is negative, the argument is the
same). We want to prove that, with exponentially small probability,
$$
\text{\rm $y_i=x_i$ for $i=1,\dots, r$; \quad   $y_{N-j}=x_{N+1-j}$ for $j=1,\dots,s+1$}.
$$
Using induction on $k$, it suffices to obtain a uniform exponentially small bound for the probability of the event $ y_{N-s-1}<x_{N-s}$ under the condition that the above equalities hold for all $i=1,\dots,r$ and for all $j=1,\dots,s$, and all the remaining elements in $Y$ are fixed.

Now we can repeat the argument for the case $k=1$. Below we only indicate necessary minor modifications. Using the shorthand notation
$$
x:=x_{N-s}, \qquad y:=y_{N-s-1}, \qquad y^*:=y_{N-s-2}
$$
we reduce the desired claim to the same two estimates \eqref{eq.A} and \eqref{eq.B}, as before.

The equality \eqref{eq.C} changes into
\begin{equation}
\label{eq_x17} \frac{\PP\big(y=x q\big)}{\PP\big(y=x\big)}=q\prod_{i=1}^{N-s-2}\frac{x
q-y_i}{x-y_i} \prod_{j=N-s}^{N-1}\frac{y_j-x q}{y_j-x},
\end{equation}
where $\PP$ now denotes the conditional distribution of $y_{N-s-1}$ given the rest. The first
product in \eqref{eq_x17} is still exponentially small as before. As for the second product,
because $x/y_j\le q$, each of its factor admits the bound
$$
\frac{y_j-x q}{y_j-x}=\frac{1-qx/y_j}{1-x/y_j}\le\frac1{1-q},
$$
so that that product does not exceed $(1-q)^{-s}$.

Next, \eqref{eq.D} changes into
\begin{equation*}
\frac{\PP\big(y=xq^n\big)}{\PP\big(y=xq\big)}
=q^{n-1}\prod_{i=1}^{N-2-s}\frac{x q^{n}-y_i}{x q-y_i}\prod_{j=N-s}^{N-1}\frac{y_j-x q^n}{y_j-x q}, \qquad n=1,2,\dots,
\end{equation*}
Here the first product is at most 1, as before, while the second product is bounded from above by $(1-q^2)^{-s}$.

Finally, \eqref{eq.E} takes the form
$$
\frac{\PP\big(y=y^*q^n \big)}{\PP\big(y=x q\big)}=\frac{|y^*|}{x}q^{n-1}
\prod_{i=1}^{N-2}\frac{|y_i|-|y^*|q^n}{|y_i|+xq}\, \prod_{j=N-s}^{N-1}\frac{y_j+|y^*|q^n}{y_j-xq}.
$$
Here we observe that, because $|y^*|\le y_j$, the second product is bounded from above by $2^s(1-q^2)^{-s}$.

These are all necessary modifications.
\end{proof}

Given an $N$--point configuration $X\in\GTe_N$, we denote by $x_{(1)},\dots,x_{(N)}$
its variational series. We also denote by  $Y$ and $y_{(1)},\dots,y_{(N-1)}$ the
random $\La^N_{N-1}(X,\,\cdot\,)$
 distributed configuration in $\GTe_{N-1}$ and its variational series, respectively. Since $Y\prec X$ with probability 1, we have $|y_{(k)}|\le|x_{(k)}|$ for every $k<N$, which in turn implies that $\log_q(|y_{(k)}|)\ge\log_q(|x_{(k)}|)$.

\begin{lemma} \label{lemma2}
Let $k=1,2,\dots$ be fixed. Then, as $C>0$ gets large, the probability of  the event $\log_q(|y_{(k)}|)-\log_q(|x_{(k)}|)\ge C$ tends to\/ $0$ uniformly on $N>k$ and $X\in\G_N$.
\end{lemma}

\begin{proof}
We closely follow  the general plan of the proof of Lemma \ref{lemma1} and keep to the notation of
that lemma. As in that lemma, we are dealing with conditional distributions. The difference is that
instead of exponentially small bounds for large $N$ we now establish uniform bounds for tails of
distributions.

We start again with  the case $k=1$ and we assume for definiteness $x_{(1)}>0$. We have to find a
uniform bound of the probability that the ratio $|y|/x$ is small.  The argument has two parts
depending on the sign of $y^*$.

Assume $y^*$ is positive, i.e.\ $y^*=x q^m$, then $y$ may take the values $x q^n$ with
$n=0,\dots,m-1$. Then we write (cf. \eqref{eq.D})
\begin{equation*}
\frac{\PP\big(y=xq^n\big)}{\PP\big(y=x\big)}=q^n\prod_{i=1}^{N-2}\frac{x q^{n}-y_i}{x-y_i}\le q^n, \qquad n=0,1,2,\dots,
\end{equation*}
This means that the probability of $y/x=q^n$ decays (as $n\to+\infty$) at least as fast as a
geometric progression with the common ratio $q$. This gives a bound for the tail probabilities that
does not depend on the length of the progression.

Assume now $y^*$ is negative. Then $y$ ranges over two infinite sequences, the positive sequence
$\{x q^n: n=0,1,\dots\}$ and the negative sequence $\{y^*q^n: n=0,1,\dots\}$.

For the positive sequence, the tails are bounded exactly as above. For the negative
sequence, the ratio $|y|/x$ equals the product $(|y^*|/x)q^n$, and this quantity can
be small due to either of the two factors, $|y^*|/x$ or $q^n$. Now we modify
\eqref{eq.E} as follows
\begin{equation*}
\frac{\PP\big(y=y^*q^n \big)}{\PP\big(y=x \big)}=\frac{|y^*|}{x}q^n
\prod_{i=1}^{N-2}\frac{|y_i|-|y^*|q^n}{|y_i|+x }\le \frac{|y^*|}{x}q^n
\prod_{i=1}^{N-2}\frac{|y_i|}{|y_i|+x}\le \frac{|y^*|}{x}q^n.
\end{equation*}
This implies that
$$
 \frac{\PP\bigl(-x q^n <y <0 \bigr)}{\PP\bigl(y=x\bigr)} < {\rm const} \cdot q^n,
$$
which gives the desired uniform bound.

The case $k\ge2$ is handled as in Lemma \ref{lemma1}: we use recursion on $k$ and slightly refine
the above arguments.
\end{proof}

\subsection{Proofs of Theorems \ref{thm1}, \ref{thm2}, \ref{thm3},  and \ref{thm4}. }

\begin{proof}[Proof of Theorems \ref{thm1} and \ref{thm2}]

\emph{Step} 1. Let  $\{X(N)\in\G_N: N=1,2,\dots\}$ be a regular sequence and
$x_{(1)}(N), \dots, x_{(N)}(N)$ denote the variational series for $X(N)$. We claim
that for every fixed $k$, the sequence $\{x_{(k)}(N)\}$ stabilizes as $N$ gets
large.

Indeed, let us fix $L>0$ which is so large that
$$
 \prod_{n=L}^\infty \left(1- \frac{e^{-c n}}{c}\right)>\frac{2}{3},
$$
where $c$ is the constant from Lemma \ref{lemma1} for this $k$. Let $Y^{(N)}\in\G_L$ be the random
configuration  with law $\La^N_L(X(N),\,\cdot\,)$ and $\xi_N$ be the probability measure on $\L$
which serves as the distribution of the $k$th term of the variational series of $Y^{(N)}$.

Due to  the choice of $L$, Lemma \ref{lemma1} implies that $\xi_N$ gives weight $>2/3$ to the point $x_{(k)}(N)\in\L$,  which in turn implies that the sum of the weights of all other points of $\L$ is $<1/3$.

On the other hand, because $\{x(N)\}$ is regular, the measures $\La^N_L(X(N),\,\cdot\,)$ weakly converge, as $N\to\infty$, to a probability measure on $\G_L$, which entails that a similar claim holds for the probability measures $\xi_N$: they weakly converge to a probability measure on $\L$.  But this may only happen if the sequence $\{x_{(k)}(N)\}$ stabilizes.

Evidently, the stable values of the terms  $x_{(k)}(N)$ form a configuration $X\in\G_\infty$.

\medskip

\emph{Step} 2.
 Conversely, let us suppose that a sequence $\{X(N)\in\G_N\}$ stabilizes to a configuration $X\in\G_\infty$. We claim that  $\{X(N)\in\G_N\}$ is regular.

Indeed, let us fix $k$ and prove that as $N\to\infty$ the
measures $\Lambda^N_k(X(N),\,\cdot\,)$ on $\G_k$ weakly converge to a probability measure $M_k$.

To do this we first show that this sequence of measures is tight and then we prove that all limiting points of this sequence are the same.

For that we fix any $\eps>0$
and choose $L>0$ such that
$$
 \prod_{n=L}^\infty \left(1-  \frac{e^{-c n}}{c}\right)>1-\eps,
$$
where $c$ is the constant from Lemma \ref{lemma1} for this $k$. Then this lemma implies that for
large $N$, with probability greater than $1-\eps$, for the variational series
$y_{(1)},y_{(2)},\dots$ of $\Lambda^N_L(X(N),\,\cdot\,)$--random point configuration it holds that
$$
y_{(1)}=x_{(1)}(N),\quad y_{(2)}=x_{(2)}(N),\quad \dots,\quad  y_{(k)}=x_{(k)}(N),
$$
where, as above, $(x_{(i)}(N))$ denotes the variational series for $X(N)$.
Now applying Lemma \ref{lemma2} and interlacing inequalities,
we see that with probability greater than $1-\eps$ the
$\Lambda^N_{k}(X(N),\,\cdot\,)$--random point configuration stays bounded away from
zero and from infinity. Since $\eps>0$ was arbitrary, this proves the tightness.

\smallskip

Now let $M_k$  be a probability measure on $\G_k$ which is a weak limit for a subsequence of $\{\Lambda^N_k(X(N),\cdot\,)$. We claim that $M_k$ is uniquely determined by $X$.

To see this we apply claim (i) of Proposition \ref{prop-links}. It says that for any partition $\nu\in\Y(k)$ one has
$$
\sum_{Y\in\G_k}\La^N_k(X(N),Y)\ti S_{\nu|k}(Y)=\ti S_{\nu|N}(X(N)).
$$
Let $N$ go to infinity along our subsequence. Then the right-hand side tends to $\ti S_{\nu}(X)$ because $X(N)$ stabilizes at $X$. As for the left-hand side, it tends to $\sum_{Y\in\G_k}M_k(Y)\ti S_{\nu|k}(Y)$: here we combine the weak convergence $\La^N_k(X(N),\,\cdot\,)\to M_k$ with the fact that the supports of all the measures from our sequence are uniformly bounded. Therefore, we arrive to the system of relations
$$
\sum_{Y\in\G_k}M_k(Y)\ti S_{\nu|k}(Y)=\ti S_{\nu}(X), \qquad \nu\in\Y(k),
$$
and the proof of  claim (ii) in Proposition \ref{prop-links} shows that this system determines $M_k$ uniquely.

The arguments of these two steps prove Theorems \ref{thm1} and \ref{thm2}.
\end{proof}

\begin{proof}[Proof of Theorem \ref{thm3}]
Let us introduce cylinder subsets in $\G_\infty$ of the form
$$
\G_\infty(x_1,\dots,x_k)=\{Z\in\G_\infty: z_{(1)}=x_1,\dots,z_{(k)}=x_k\},
$$
where $x_1,\dots,x_k$ is a given finite sequence of points of $\L$.  For $L\ge k$, we  denote by $\G_L(x_1,\dots,x_k)$ the similar cylinder subset in $\G_L$.

Next, let $X\in\G_\infty$ be fixed and $(x_{(1)}, x_{(2)}, \dots)$ be the corresponding variational
series. Theorem \ref{thm1}, that has been just proved, assigns to $X$ a coherent system
$M^{(X)}=\{M^{(X)}_L: L=1,2,\dots\}$. In these terms, the claim of Theorem  \ref{thm3} can be
rephrased in the following way:  For every fixed $k=1,2,\dots$, the mass given by $M^{(X)}_L$ to
the subset $\G_L(x_{(1)},\dots,x_{(k)})\subset\G_L$ tends to $1$ as $L\to\infty$.

Let us prove the last statement.  Given an arbitrary  $\eps>0$ we choose $L_0$ so large that
$$
\prod_{n=L_0}^\infty(1-\frac{e^{-c n}}{c})\ge 1-\eps,
$$
where $c$ is the constant from Lemma \ref{lemma1} that correspond to our fixed $k$. Next, let
$X(N)\in\G_N$ stand for the configuration formed by $x_{(1)},\dots,x_{(N)}$.  By virtue of Lemma
\ref{lemma1},  for any $N>L\ge L_0$, the mass assigned by the measure $\La^N_L(X(N),\,\cdot\,)$ to
the subset $\G_L(x_{(1)},\dots,x_{(k)}$ is $\ge1-\eps$.

On the other hand, Theorem \ref{thm1} says us that our sequence $\{X(N)\}$ is regular and represents the coherent system $M^{(X)}$, so that, as $N\to\infty$, the measures $\La^N_L(X(N),\,\cdot\,)$ converge to the measure $M^{(X)}_L$. It follows that $M^{(X)}_L(\G_L(x_{(1)},\dots,x_{(k)})\ge1-\eps$ for any $L\ge L_0$.

This completes the proof.
\end{proof}

\begin{proof}[Proof of Theorem \ref{thm4}]
We will prove the claims of the theorem in the reverse order.

\emph{Step} 1. The correspondence $X\mapsto M^{(X)}$ is an injective map $\G_\infty\to\varprojlim\M(\G_N)$. We claim that the image of $\G_\infty$ is a Borel subset and the map is a Borel isomorphism onto this image.

Indeed, observe that both spaces $\G_\infty$ and $\varprojlim\M(\G_N)$ are obviously standard Borel  spaces. Therefore, by virtue of an abstract theorem (see Mackey \cite[Theorem 3.2]{Mackey}), it suffices to check that the map in question is Borel. By the definition of the Borel structure in $\varprojlim\M(G_N)$ this amounts to verifying that the functions on $\G_\infty$ of the form
$$
f_{K,Y}(X):=M^{(X)}_K(Y), \qquad K=1,2,\dots,\quad Y\in\G_K,
$$
are Borel.

Now let $X(N)\in\G_N$ be composed from the first $N$ terms of the variational series of $X$. By Theorem \ref{thm1},
$$
M^{(X)}_K(Y)=\lim_{N\to\infty}\La^N_K(X(N),Y).
$$
It follows that $f_{K,Y}(X)$ is a pointwise limit of cylinder functions and hence is Borel.

\emph{Step} 2. Let $X\in\G_\infty$ be arbitrary and let $\sigma$ be the boundary
measure of $M^{(X)}$, which is a probability Borel measure on the minimal boundary
(Theorem \ref{thm-boundary}). We are going to prove that $\sigma$ is the
delta-measure at $X$, which will imply that  $M^{(X)}$ is extreme.

By Theorem \ref{Theorem_minimal_in_Martin} the minimal boundary is contained in the
Martin boundary, which can be identified with $\G_\infty$ (Theorem \ref{thm1}).
Next, by virtue of step 1, we may interpret $\sigma$ as an element of
$\M(\G_\infty)$. Therefore, we may write
$$
 M^{(X)}=\int_{Z\in\G_\infty} M^{(Z)} \sigma(dZ)
 $$
 and, consequently,
  $$
 M^{(X)}_L=\int M^{(Z)}_L \sigma(dZ), \qquad L=1,2,\dots\,.
 $$

 We are going to show that $\sigma$ must be the delta-measure at $X$, which will imply that $X$ is extreme. To this end it suffices to prove that for every $k=1,2,\dots$, the measure $\sigma$ is concentrated on $\G_\infty(x_{(1)},\dots,x_{(k)})$, where $(x_{(i)})$ is the variational series for $X$.

The argument based on Lemma \ref{lemma1} and used in the proof of Theorem \ref{thm3} shows that  for any $\eps>0$ there exists $L=L(k, \eps)$ such that for any $Z\in\G_\infty$ one has
 $$
 M^{(Z)}_L(\G_L(z_{(1)},\dots,z_{(k)}))\ge 1-\eps,
 $$
 where $(z_{(1)}, z_{(2)}, \dots)$ denotes the variational series for $Z$.

 Assume now that there exists a cylinder subset $\G_\infty(x_1,\dots,x_k)\subset\G_\infty$ which is distinct from $\G_\infty(x_{(1)},\dots,x_{(k)})$ and has a strictly positive $\sigma$-mass:
 $$
 \sigma(\G_\infty(x_1,\dots,x_k))=\de>0.
 $$
 Then for $L=L(k,\eps)$ we have
 $$
 \int_{Z\in\G_\infty} M^{(Z)}_L(\G_L(x_1,\dots,x_k))\sigma(dZ)\ge \int_{Z\in\G_\infty(x_1,\dots,x_k)} M^{(Z)}_L(\G_L(x_1,\dots,x_k))\sigma(dZ)\ge\de(1-\eps).
 $$

 On the other hand, the first integral is equal to
 $$
 M^{(X)}_L(\G_L(x_1,\dots,x_k))\le\eps,
 $$
 so that $\eps\ge\de(1-\eps)$.  Choosing $\eps$ sufficiently small we get a contradiction, which completes the proof.
 \end{proof}

\subsection{Approximation of boundary measures}

Recall that the abstract Theorem \ref{thm-boundary} assigns to an arbitrary coherent system
$M=\{M_N\}\in\varprojlim\M(\Om_N)$ a probability measure $\sigma$ on the minimal boundary $\Om$. It
is tempting to regard this boundary measure as a limit of the sequence $\{M_N\}$. And indeed, in a
number of concrete cases this informal statement can be turned into a formal one (see e.g.
Olshanski \cite[Theorem 10.2]{Ols-JFA}. Here we show how to do this in our situation.

In the theorem below $M=\{M_N\}\in\varprojlim\M(\G_N)$ is an arbitrary coherent
system and $\sigma$ is its boundary measure. By virtue of Theorem \ref{thm4} we may
regard $\sigma$ as a probability Borel measure on $\G_\infty$. Therefore, we may put
all the probability measures under considerations, that is, $M_N$'s and $\sigma$, on
a common space --- the space $\bar\G_\infty$ introduced in Subsection \ref{sect3.2}.
Recall that it is a locally compact topological space.

\begin{theorem}\label{thm-approximation}
Let $\{M_N\}$ and $\sigma$ be as above. Then $M_N\to \sigma$ in the weak topology of the space $\M(\bar\G_\infty)$.
\end{theorem}

\begin{corollary} \label{cor-approximation}
In particular, for every $X\in\G_\infty$, the measures $M^{(X)}_N=\La^\infty_N(X,\,\cdot\,)$ weakly converge to the delta measure at $X$.
\end{corollary}

\begin{proof}[Proof of Theorem \ref{thm-approximation}]
Let $C(\bar\G_\infty)$ be the space of bounded continuous functions on
$\bar\G_\infty$. We have to prove that
$$
\lim_{N\to\infty}\langle f, M_N\rangle=\langle f, \sigma\rangle, \qquad f\in C(\bar\G_\infty),
$$
where the angular brackets denote the pairing between functions and measures.

We have
$$
\langle f, M_N\rangle=\int_{\G_\infty} \langle f, M^{(X)}_N \rangle \sigma(dX).
$$
By Lebesgue's dominated convergence theorem, it suffices to prove that
$$
\lim_{N\to\infty}\langle f, M^{(X)}_N \rangle=f(X), \qquad X\in\G_\infty, \quad f\in C(\bar\G_\infty).
$$
We will derive this claim (which is just the claim of Corollary
\ref{cor-approximation}) from the law of large numbers for the measures $M^{(X)}_N$
established in Theorem \ref{thm3}.

We use the notation of that theorem. Observe that there exists a compact set in $\bar\G_\infty$ containing $X$ and the supports of all measures $M^{(X)}_N$: for instance, one can take the set
$$
\{Z\in\bar\G_\infty: |z_{(1)}|\le |x_{(1)}|\}.
$$
On this set, the function $f(Z)$ can be approximated, in the uniform norm, by cylinder functions which are polynomials in finitely many variables $z_{(1)}, z_{(2)},\dots$\,. Therefore we may assume that $f$ itself is such a polynomial. But then the desired result follows from Theorem \ref{thm3}.
\end{proof}

\begin{definition} \label{Definition_cor_function}
Let $\sigma$ be an arbitrary Borel probability measure on $\bar\G_\infty$. The
\textit{$n$-particle correlation function} of $\sigma$, denoted by
$\rho^n_{\sigma}(x_1,\dots,x_n)$, assigns to an arbitrary $n$-tuple of pairwise
different points $x_1,\dots,x_n\in\L$ the probability of the event that the
$\sigma$-random configuration contains all these points.
\end{definition}

The next corollary will be used in Section \ref{sect5}.

\begin{corollary}\label{corrfunctions}
Let $\{M_N\}$ be a coherent system and $\sigma$ be its boundary measure.  For every fixed
$n=1,2,\dots$ and any pairwise different $x_1,\dots,x_n\in\L$,
$$
\lim_{N\to\infty}\rho^n_{M_N}(x_1,\dots,x_n)=\rho^n_{\sigma}(x_1,\dots,x_n).
$$
\end{corollary}

\begin{proof}
Consider the function $f(Z)$ on $\bar\G_\infty$ which equals $1$ if $Z$ contains all the points $x_1,\dots,x_n$, and $0$ otherwise. We have
$$
\rho^n_{M_N}(x_1,\dots,x_n)=\langle f, M_N\rangle, \quad
\rho^n_{\sigma}(x_1,\dots,x_n)=\langle f, \sigma\rangle.
$$

On the other hand, $f$ is locally constant and hence continuous, as is immediately
seen from Proposition \ref{prop-topology}. Then the desired result follows from
Theorem \ref{thm-approximation}.
\end{proof}

\section{The $q$-zw-measures} \label{sect4}

\subsection{Definition of the measures}

We are going to introduce a family of probability measures on the levels $\G_N$ of
the extended Gelfand-Tsetlin graph $\G$. This family is the main object of study in
the present paper. The measures of the family depend on a quadruple  $(\al$, $\be$,
$\ga$, $\de)$ of parameters, which are complex or real numbers subject to some
conditions. We first give the formula for the measures and then discuss the
constraints on the parameters.

In what follows we use the conventional notation from $q$-analysis (see Gasper--Rahman \cite{GR})
$$
(a;q)_\infty:=\prod_{m=0}^\infty(1-aq^m), \qquad (a;q)_n:=\frac{(a;q)_\infty}{(aq^n;q)_\infty}=\prod_{m=0}^{n-1}(1-aq^m), \quad n=0,1,2,\dots\,.
$$
 For $N=1,2,\dots$ and $X=(x_1<\dots<x_N)\in \GTe_N$ we define
 \begin{equation} \label{eq_measure_def}
  \M_N^\alde(X)= \frac{1}{\Z_N(\alde)} \prod_{i=1}^N w^\alde_N(x_i) \prod_{1\le i<j\le N}
  (x_j-x_i)^2,
 \end{equation}
 where
 $$
 w^\alde_N(x)=|x|\frac{(\al x, \be x; q)_\infty}{(\ga q^{1-N}x, \de q^{1-N}x; q)_\infty}, \qquad x\in\L,
 $$
 and $\Z_N(\alde)$ is a normalization constant,
 $$
 \Z_N(\alde)=\sum_{X\in\G_N} \prod_{i=1}^N w^\alde_N(x_i) \prod_{1\le i<j\le N}
  (x_j-x_i)^2
 $$
 (a closed expression for $\Z_N(\alde)$ is given in Section \ref{Section_orth_polynomials} below).  Of course, for $N=1$ the product over $i<j$ is missing.

The conditions on the parameters $(\alde)$ must guarantee the correctness of the
definition. That is, $w^\alde_N(x)$ has to be a nonnegative real number for all $N$
and all $x\in\L$, and $\Z_N(\alde)$ has to be finite for all $N$.

There are two cases, the \textit{nondegenerate} and \textit{degenerate} ones.  In
the nodegenerate case the quantity $w^\alde_N(x)$ is always strictly positive, and
in the degenerate case it may vanish for some $x\in\L$.

We are mostly interested in the nondegenerate case, which is described with the help of the
following lemma.

\begin{lemma}
Assume $(a,b)$ is a pair of nonzero complex numbers such that one of the following two conditions
holds:

{\rm(i)}  either $a$ and $b$ are nonreal and complex-conjugate

{\rm(ii)} or $a$ and $b$ are real and such that $a^{-1}$ and $b^{-1}$ are contained in an open interval between two neighboring points of the lattice $\L$.

Then the quantity  $(ax,bx;q)_\infty$ is real and strictly positive for all $x\in\L$.
\end{lemma}

\begin{proof}
Condition (i) guarantees that $(ax;q)_\infty$ and $(bx;q)_\infty$ are nonzero and complex-conjugate. Condition (ii) guarantees that these two quantities are real, nonzero, and of the same sign.
\end{proof}

One can easily prove the converse statement: if $a$ and $b$ are nonzero and such that $(ax, bx;q)_\infty$ is real and strictly positive for every $x\in\L$, then either (i) or (ii) holds true.

It is evident that any of the conditions (i), (ii) implies that $ab$ is real and strictly positive. One more evident observation: if $(a,b)$ satisfiles (i) or (ii), then the same holds true for $(a q^n,b q^n)$ with $n\in\ZZ$.

\begin{definition}\label{def.A}
We say that a quadruple $(\alde)$ of parameters is \textit{admissible} and \textit{nondegenerate} if both pairs $(\al,\be)$ and $(\ga,\de)$ satisfy the assumptions on the previous lemma and, moreover, $\ga\de q>\al\be$.
\end{definition}

\begin{proposition}
If $(\alde)$ is admissible and nondegenerate in the sense of the previous
definition, then the definition \eqref{eq_measure_def} of the probability measures
$M^\alde_N$ makes sense.
\end{proposition}

\begin{proof}
The conditions on the parameters guarantee that $w^\alde_N(x)>0$ for every $N$ and all $x\in\L$, so
the question to settle is whether the sum, which we denoted by $\Z_N(\alde)$,  is finite. We will
prove that
\begin{equation}\label{eq_x18}
\sum_{x\in\L}w_N^\alde(x)\, |x|^{i}<\infty, \qquad i=0,1,\dots,2N-2,\quad N=1,2,\dots,
\end{equation}
from which the desired conclusion  $\Z_N(\alde)<\infty$ evidently follows.

Note that $|x|^{-1} w_N^\alde(x)\to 1$ as $x\in\L$ approaches $0$, which implies the convergence of
the part of the sum \eqref{eq_x18} corresponding to all $x$ in a finite neighborhood of the origin.

 Let us examine the asymptotics of $w^\alde_N(x)$ as $x$ goes to plus or minus infinity along
$\L$. It is convenient to introduce the shorthand notation
$$
(a,b,c,d)=(\al,\be,\ga q^{1-N}, \de q^{1-N}), \quad w(x)=|x| \frac{(ax,bx;q)_\infty}{(cx,dx;q)_\infty}.
$$

Every point $x\in\L$ has the form $x=\z_\pm q^{-n}$, where $n\in\ZZ$. In this
notation, $x$ goes to infinity means that $n$ goes to $+\infty$. Then we have
\begin{equation}\label{asymptotics}
w(\z_\pm q^{-n})=|\z_\pm|\frac{(a\z_\pm, b\z_\pm;q)_\infty}{(c\z_\pm, d\z_\pm;q)_\infty}\, \frac{(a^{-1}\z_\pm^{-1}, b^{-1}\z_\pm^{-1})_n}{(c^{-1}\z_\pm^{-1}, d^{-1}\z_\pm^{-1};q)_n}\, \left(\frac{ab}{cdq}\right)^n\sim \const \left(\frac{ab}{cdq}\right)^n.
\end{equation}

It follows that
$$
w^\alde_N(\z_\pm q^{-n}) x^{2-2N}\sim \const \left(\frac{\al\be}{\ga\de q}\right)^n, \qquad n\to+\infty.
$$
This shows that needed sum is indeed finite.
\end{proof}

The degenerate case splits into several subcases.  Here is one of them.

\begin{example}\label{ex.A}
Assume $\al$ and $\be$ are real, $\al<0<\be$, and $\al^{-1}\in\L$,  $\be^{-1}\in\L$. As for $\ga$ and $\de$, assume that they are nonreal and complex-conjugate. Then $w_N^\alde(x)$ is strictly positive if $\al^{-1}q\le x\le \be^{-1}q$, and vanishes otherwise.  In this case, the measures are also correctly defined.
\end{example}

Unless otherwise stated, in what follows we will assume that $(\alde)$ is admissible and
nondegenerate in the sense of Definition \ref{def.A}. However, our results are automatically
extended to various variants of the degenerate cases. In particular, one may choose the parameters
as in Example \ref{ex.A} above.

\subsection{Pseudo big q-Jacobi  polynomials}

\label{Section_orth_polynomials}

The definition \eqref{eq_measure_def} of the measure $M^\alde_N$ on $\G_N$ fits into
the general scheme of \textit{orthogonal polynomial ensembles}, see, e.g., K\"onig's
survey paper \cite{Konig}. In our situation, the orthogonal polynomials in question
are those associated with the weight function $w_N^\alde(x)$ on $\L$. These
polynomials are not mentioned in the encyclopedic paper by Koekoek-Swarttouw
\cite{KS}. Fortunately, they appeared in recent Koornwinder's paper
\cite{Koo-Addendum}, which is an addendum to \cite{KS}).  Below we collect the
necessary information from \cite{Koo-Addendum}.

\begin{definition}
The \textit{pseudo big q-Jacobi polynomials} with parameters $(a,b,c,d)$ are defined by
$$
P_n(x)=P_n(x; a,b,c,d):=
{}_3\phi_2\left(\begin{matrix} q^{-n}, \; cda^{-1}b^{-1}q^{n+1}, \; cx\\
cb^{-1}q, \; ca^{-1}q\end{matrix}\bigg| q\right), \qquad n=0,1,2,\dots,
$$
where we use the conventional notation
$$
{}_3\phi_2\left(\begin{matrix} A, \; B, \; C\\
D, \; E\end{matrix}\bigg |Z\right):=\sum_{n=0}^\infty\frac{(A;q)_n(B;q)_n(C;q)_n}{(D;q)_n(E;q)_n(q;q)_n} Z^n.
$$
\end{definition}

The key property of the polynomials $P_n(a,b,c,d)$ is that they are orthogonal on $\L$ with weight
$$
w(x)=w(x; a,b,c,d)=|x|\frac{(ax;q)_\infty (bx;q)_\infty}{(cx;q)_\infty (dx;q)_\infty}, \quad x\in\L.
$$
Here we tacitly assume that the parameters are such that the weight is nonnegative.  We  are mainly interested in the nondegenerate case, when the weight is nowhere vanishing. Then, as is seen from \eqref{asymptotics},  it possesses only finitely many moments, so that the corresponding family of orthogonal polynomials is finite, too. The maximal degree of orthogonal polynomials is the largest integer $n$ such that $\frac{cdq}{ab}>q^{-2n}$.

It is worth noting that the  polynomials $P_n(x; a,b,c,d)$ do not depend on the extra parameters
$\z_-,\z_+$, which enter the definition of the lattice $\L$, the support of the measure. So, fixing
$(a,b,c,d)$ we still dispose of a two-parameter family of different orthogonality measures.  Of
course, the non-uniqueness of the orthogonality measure is readily explained, because  the number
of orthogonal polynomials is finite. However, the possibility to explicitly exhibit a whole family
of orthogonality measures is a remarkable fact. It also shows that the role of parameters
$(\z_-,\z_+)$ is less important than that of parameters  $(a,b,c,d)$.

The polynomials $P_n(x)$ are closely related to the classical  \textit{big $q$-Jacobi
polynomials}, cf. Andrews-Askey \cite{AndrewsAskey}, Koekoek-Swarttouw \cite{KS}, Ismail \cite{Ismail}, Koornwinder \cite{Koo}, defined via
$$
\mathcal P_n(x; A, B,C)=
{}_3\phi_2\left(\begin{matrix} q^{-n}, \; ABq^{n+1}, \; x\\
Aq, \; Cq\end{matrix}\bigg| q\right).
$$
Namely, let us assume that the two sets of parameters are related by
\begin{equation} \label{eq_parameters_relation}
A=\frac cb, \quad B=\frac da, \quad C=\frac ca.
\end{equation}
Then the comparison of the formulas above shows that
\begin{equation}
\label{eq_relation_Jac}
P_n(x; a,b,c,d)=\mathcal P_n(cx; A,B,C).
\end{equation}
Looking at \eqref{eq_relation_Jac} it might seem that our polynomials merely coincide with
classical $\mathcal P_n$ up to a rescaling of the variable. This is not quite true, as the
admissible domains of parameters ($a$,$b$,$c$,$d$ and $A$, $B$, $C$) for these two families of
polynomials are different and the supports of their orthogonality measures are also different. A
more precise point of view is that the polynomials $P_n$ and $\mathcal P_n$ are \emph{analytic
continuations} (in parameters) of each other. Nevertheless, relation \eqref{eq_relation_Jac} makes
it possible to extract a number of necessary formulas for the polynomials $P_n$ from well-known
formulas for the big $q$-Jacobi polynomials $\mathcal P_n$.

\smallskip

Note that the above expression for $P_n(x;a,b,c,d)$ is evidently symmetric under transposition
$a\leftrightarrow b$, and there is also a hidden symmetry under transposition $c\leftrightarrow d$:
namely, the \textit{monic} polynomials (which are obtained when we divide $P_n$ by its top degree
coefficient) are symmetric under $c\leftrightarrow d$.

Let  $h_n(a,b,c,d)$ denote the squared norm of $P_n(x;a,b,c,d)$:
$$
h_n(a,b,c,d)=\sum_{x\in\L} (P_n(x;a,b,c,d))^2 w(x;a,b,c,d).
$$
Here is an explicit expression for this quantity (see Koornwinder \cite[Section 14.5]{Koo-Addendum} and references therein):
\begin{equation}
\frac{h_n(a,b,c,d)}{h_0(a,b,c,d)}=(-1)^n\left(\frac{c^2}{ab}\right)^n q^{n(n-1)/2} q^{2n}\,
\frac{(q,qd/a,qd/b;q)_n}{(qcd/(ab),qc/a,qc/b;q)_n}\,
\frac{1-qcd/(ab)}{1-q^{2n+1}cd/(ab)}
\end{equation}
and
\begin{equation}
h_0(a,b,c,d)=\z_+\, \frac{(q,a/c,a/d,b/c,b/d;q)_\infty}{(ab/(qcd);q)_\infty}\,
\frac{\theta_q(\z_-/\z_+,cd\z_-\z_+)}{\theta_q(c\z_-,d\z_-,c\z_+,d\z_+)}\,,
\end{equation}
where the theta function $\theta_q(u)$ of a single argument $u$ is defined by
$$
\theta_q(u)=(u, q/x;q)_\infty
$$
and we use the shorthand notation
$$
(u_1,\dots,u_m;q)_n=(u_1;q)_n\dots(u_m;q)_n, \qquad \theta_q(u_1,\dots,u_m):=\theta_q(u_1)\cdots\theta_q(u_m).
$$

We also need the top degree coefficient of $P_n(x;a,b,c,d)$, which we denote by
$k_n(a,b,c,d)$. It is readily obtained from the definition of the polynomial:
\begin{equation}
\label{eq_leading_term} k_n(a,b,c,d)=c^n \frac{(cd q^{n+1}/(ab); q)_n}{(cq/b,
cq/a;q)_n}.
\end{equation}

\begin{proposition}
In this notation we have
$$
\Z_N(\alde)=\prod_{n=0}^{N-1}\frac{h_n(\al,\be,\ga q^{1-N}, \de q^{1-N})}{k^2_n(\al,\be,\ga q^{1-N}, \de q^{1-N})}.
$$
\end{proposition}

\begin{proof}
This is a general fact, which holds for any orthogonal polynomial ensemble (see e.g. the computation in K\"onig \cite[Section 2.7]{Konig}).
\end{proof}

It is a remarkable fact that there exists a closed expression for the sum
$$
h_0(a,b,c,d) :=\sum_{x\in\L}w(x; a,b,c,d).
$$
This is a true $q$-analog of famous Dougall's formula \cite{Dougall} for the bilateral hypergeometric series ${}_2H_2$ at $1$. It seems that no closed expression exists for the sum of the weights $w(x; a, b, c, d)$ over the positive part of $\L$, which would correspond to a naive extension of Dougall's formula (see Askey \cite{Askey} and Bailey \cite{Bailey}).

\subsection{Coherency via orthogonal polynomials.}

The aim of this section is to prove the following theorem.

\begin{theorem} \label{theorem_coherency}
 Assume that the quadruple $(\alde)$ of parameters is admissible and nondegenerate in the sense of Definition\/ \ref{def.A}. Then the measures
 $\M_N^{\al,\be,\ga,\de}$ on $\GTe_N$, $N=1,2,\dots$ given by \eqref{eq_measure_def}
 form a coherent system on the graph $\G$.
\end{theorem}

Before proceeding to the proof we need to perform some preparatory work. In the argument below we
apply a trick consisting in introducing some ``virtual particles''. We learnt it from Alexei
Borodin; see Borodin--Ferrari-Pr\"ahofer--Sasamoto \cite[Lemma 3.4]{BFPS}.

If $x,y\in\L$, we write $x\TI y$ if either $y$ is positive and $x<y$ or if $y$ negative and $x\le
y$. Next, introduce a kernel on $\L\times(\L\cup\{+\infty\})$ by setting
$$
 A(x,y)=
 \begin{cases}
  1,& x\TI y\\ 0,&\text{otherwise }
 \end{cases}
 $$
 with the understanding that $A(x,+\infty)=1$ for every $x\in\L$.

\begin{lemma} \label{Prop_interlacing_as_det}
Let $X=(x_1<\dots<x_{N+1})\in\G_{N+1}$, $Y=(y_1<\dots<y_N)\in\G_N$, and $y_{N+1}:=+\infty$.

 Then
$$
 \det [A(x_i,y_j)]_{i,j=1}^{N+1}=\begin{cases} 1,& X\succ Y,\\ 0, & \text{otherwise}
 \end{cases},
$$
\end{lemma}

\begin{proof}
 As explained in Subsection \ref{sect2.1}, the points $y_0:=-\infty$, $y_1$, \dots, $y_N$, $y_{N+1}=+\infty$ divide $\L$ into a disjoint union
  into $N+1$ intervals $\ti I(y_{i-1}, y_i)$, $i=1,\dots,N+1$,
 with the agreement that the $i$th interval includes its right endpoint $y_i$ if and only if $y_i<0$, and includes its left endpoint $y_{i-1}$ if and only if $0<y_{i-1}$.

If two of the points $x_1,\dots,x_{N+1}$ belong to the same
  interval, then the corresponding two rows of the matrix $A(x_i,y_j)$ are the same and the
  determinant vanishes. Thus, for non-zero determinant, all the points $x_1,\dots,x_{N+1}$ need to belong to different intervals, so that $x_1\in I_1,\dots, x_{N+1}\in I_{N+1}$. A case-by-case check shows that this just means $X\succ Y$.

 Further, in this case, the matrix $A(x,y)$ is triangular
  with $1$s on the diagonal and, thus, its determinant is one.
\end{proof}

\begin{lemma}\label{prop_poly_relation}
Let us set
\begin{gather*}
w(x)=w(x;a,b,c,d), \quad P_n(x)=P_n(x;a,b,c,d), \\
 w^*(x)=w(x; a,b, cq^{-1}, dq^{-1}), \quad P^*_{n+1}(x)=P_{n+1}(x; a,b, cq^{-1}, dq^{-1}),
 \end{gather*}
where $n\ge 0$ is such that $\frac{cdq}{ab}>q^{-2n}$.  Then for every $y\in\L$ one
has
\begin{equation}\label{eq_x10}
\sum_{x\in\L:\, x\TI y} w^*(x) P^*_{n+1}(x)=\frac{cq}{(b-c)(a-c)}\cdot
  \frac{w(y)P_n(y)}{|y|}.
\end{equation}
\end{lemma}

\begin{proof}
 Observe that as $y\to-\infty$ the left--hand side of \eqref{eq_x10} tends to $0$.
 The right--hand side also tends to zero due to the condition on $n$.
  Similarly, in the limit $y\to+\infty$, the left--hand side of
 \eqref{eq_x10} becomes the scalar product of $P^*_{n+1}$ with constant function and
 hence vanishes and so is the right--hand side.

Therefore, it suffices to check that the difference of \eqref{eq_x10} at $y\in\L$
and at $yq$ is a valid identity.

If $y$ is negative, then $x\TI y$ means $x\le y$ and the difference on the left is
$$
-w^*(yq)P^*_{n+1}(yq).
$$
If $y$ is positive, then $x\TI y$ means $x<y$ and the difference on the left is
$$
w^*(yq)P^*_{n+1}(yq).
$$
Therefore, in both cases the desired identity can be written as
$$
y\, w^*(yq)P^*_{n+1}(yq)=\frac{cq}{(b-c)(a-c)}
\bigl(w(y)P_n(y)-q^{-1}w(yq)P_n(yq)\bigr)
$$
or equivalently
$$
y
P^*_{n+1}(yq)=\frac{cq}{(b-c)(a-c)}\left(\frac{w(y)}{w^*(yq)}P_n(y)-q^{-1}\frac{w(yq)}{w^*(yq)}P_n(yq)\right).
$$

We claim that the latter relation follows from the backward shift relation for the
big $q$-Jacobi polynomials (Koekoek and Swarttouw \cite[(3.5.8)]{KS})
\begin{multline*}
(1-A)(1-C)u \mathcal P_{n+1}(u; Aq^{-1}, Bq^{-1}, Cq^{-1})\\
=(u-A)(u-C)\mathcal P_n(u;  A,B,C) -A(u-1)(Bu-C)\mathcal P_n(uq; A,B,C).
\end{multline*}

Indeed, let us set $u=cy$ and compare the two relations using the connection between the two families of polynomials.

We have
$$
P^*_{n+1}(yq)=P_{n+1}(yq, a,b, cq^{-1}, dq^{-1})=\mathcal P_{n+1}(cy; Aq^{-1}, Bq^{-1}, Cq^{-1}).
$$
This shows that  the left hand sides of our relations differ by a scalar factor.

Let us examine now the right-hand sides. We have
$$
P_n(y)=\mathcal P_n(u; A,B,C), \quad P_n(yq)=\mathcal P_n(uq; A,B,C)
$$
and
$$
w(y)=|y|\frac{(ay, by;q)_\infty}{(cy,dy;q)_\infty}, \quad w(yq)=q|y|\frac{(ayq, byq;q)_\infty}{(cyq,dyq;q)_\infty}, \quad w^*(yq)=q|y|\frac{(ayq, byq;q)_\infty}{(cy,dy;q)_\infty}.
$$
It follows that
$$
\frac{w(y)}{w^*(yq)}=q^{-1}(1-ay)(1-by), \quad q^{-1}\frac{w(yq)}{w^*(yq)}=q^{-1}(1-cy)(1-dy).
$$

On the other hand,
$$
(u-A)(u-C)=\frac{c^2}{ab}(1-ay)(1-by), \qquad A(u-1)(Bu-C)=\frac{c^2}{ab}(1-cy)(1-dy),
$$
which shows that the right-hand sides of our relations differ by the same constant factor as above.

This completes the proof.
\end{proof}

\begin{proof}[Proof of Theorem \ref{theorem_coherency}]
We prove that
$$
\M^{\al,\be,\ga,\de}_{N+1}\La^{N+1}_N=\M^{\al,\be,\ga,\de}_N, \qquad N=1,2,\dots\,.
$$
Since $\M^{\al,\be,\ga,\de}_{N+1}$ and $\M^{\al,\be,\ga,\de}_N$ are probability
measures, it suffices to prove a formally weaker claim that the both sides differ by
a nonzero constant factor. This will substantially simplify the computations,
because we can ignore  the cumbersome expression for the normalization factor
entering the definition of our measures. On the other hand, if one reconstructs all
the normalization factors in the present proof, then this gives an independent check
of the explicit formula for $\Z_N(\alde)$.

Throughout the proof  we denote by ``$\const$'' a (possibly varying) nonzero constant factor whose exact value is not relevant for us.
We keep to the notation of Lemma \ref{prop_poly_relation}, where we set
$$
(a,b,c,d)=(\al,\be,\ga q^{1-N}, \de q^{1-N}).
$$
Taking into account the definition of the measures and the links, the desired relation can be written as
$$
\sum_{X:\, X\succ Y}\det[w^*(x_i)x_i^{k-1}]_{i,k=1}^{N+1}=\const \det[|y_j|^{-1} w(y_j)y_j^{k-1}]_{j,k=1}^N, \qquad \forall Y\in\G_N
$$
or equivalently as
$$
\sum_{X:\, X\succ Y}\det[w^*(x_i)P^*_{k-1}(x_i)]_{i,k=1}^{N+1}=\const \det[|y_j|^{-1} w(y_j)P_{k-1}(y_j)]_{j,k=1}^N, \qquad \forall Y\in\G_N.
$$

Next, applying Lemma \ref{Prop_interlacing_as_det} we may write the left-hand side as
$$
\sum_{x_1<\dots<x_{N+1}}\det[w^*(x_i)P^*_{k-1}(x_i)]_{i,k=1}^{N+1}\det[A(x_i, y_j)]_{i,j=1}^{N+1}.
$$
Applying the Cauchy-Binet formula and taking into account the definition of the kernel $A(x,y)$ we write the above sum as a single determinant
$$
\det[B(j,k)]_{j,k=1}^{N+1}, \qquad B(j,k):=\sum_{x: \,x\TI y_j}w^*(x)P^*_{k-1}(x).
$$
Now we use the fact that $y_{N+1}=+\infty$. It follows that
$$
B(N+1,k)=\sum_{x\in\L}w^*(x)P^*_{k-1}(x).
$$
Because of the orthogonality of the polynomials $P^*_n$, this expression vanishes unless $k=1$, when it is a positive constant. Therefore, our determinant of order $N+1$ is reduced, up to a scalar factor, to the minor of order $N$ which is obtained by removing the $(N+1)$th row and the first column. This minor has the form
$$
\det\left[\sum_{x:\, x\TI y_j}w^*(x) P^*_k(x)\right]_{j,k=1}^N.
$$

Finally we apply Lemma \ref{prop_poly_relation}. It says that this determinant is equal to
$$
\const\det\left[|y_j|^{-1}w(y_j) P_{k-1}(y_j)\right]_{j,k=1}^N,
$$
which is precisely what we want.
\end{proof}

\begin{corollary}
 Assume that the quadruple $(\alde)$ of parameters is admissible and nondegenerate in the sense of Definition\/ \ref{def.A}. Then, by virtue of the abstract Theorem \ref{thm-boundary}, the coherent system
 $\{\M_N^\alde: N=1,2,\dots\}$ afforded by Theorem \ref{theorem_coherency} gives rise to a probability measure on the space $\G_\infty$.
\end{corollary}

\begin{definition}\label{def-boundary-measures}
The probability measure from the above corollary will be denoted by $\M^\alde_\infty$ and called a \textit{boundary $q$-zw-measure}.
\end{definition}

\section{Correlation functions of the boundary $q$-zw-measures}\label{sect5}

Our definition of the boundary $q$-zw-measures  $\M^\alde_\infty$ (Definition
\ref{def-boundary-measures} above) is non-constructive in the sense that it relies on an abstract
existence theorem. So the question arises whether it is possible to obtain a more concrete
description of these measures. There is no hope that they can be given by a density with respect to
some evident reference measure like Lebesgue measure on the line. This is a standard situation for
measures on infinite-dimensional spaces. In the case of the ordinary Gelfand--Tsetlin graph, the
boundary zw-measures were described in Borodin--Olshanski \cite{BO-AnnMath} via their correlation
functions. We follow the same approach and find in this section the correlation functions of
$\M^\alde_\infty$.

We recall that the definition of correlation functions adapted to our concrete situation was given in
Definition \ref{Definition_cor_function}.

The correlation functions of $\M_\infty^{\alpha,\beta,\gamma,\delta}$ are computed
in terms of certain basic hypergeometric functions that we now introduce:
\begin{multline} \label{eq_F_function}
\F_\rr(x)=\F^{\alpha,\beta,\gamma,\delta}_\rr(x)
\\= \sqrt{|x| \dfrac{ (x\al,x\be;q)_\infty}{\theta_q(x\ga) \theta_q(x\de)}}  \cdot x^{1-\rr}\, \dfrac{\left(\dfrac\be\ga q^{\rr-1}, \dfrac{q^\rr}{\de
x};q\right)_\infty}{\left(\dfrac{\al\be}{\ga\de}q^{2\rr-2};q\right)_{\infty} } \cdot
\, {}_2\phi_1\left(\begin{matrix}
\dfrac{\al q^{\rr-1}}\de, \, \dfrac q{\be x}\\
\dfrac{q^\rr}{\de x}\end{matrix}\Bigg| \frac{q^{\rr-1}\be}\ga\right) ,\quad
\rr\in\mathbb Z,
\end{multline}
where
$$
{}_2\phi_1\left(\begin{matrix}
A, \, B\\
C\end{matrix}\bigg| Z\right)=\sum_{n=0}^\infty\frac{(A;q)_n(B;q)_n}{(C;q)_n(q;q)_n}Z^n.
$$

\begin{remark}
The function $\F^{\alpha,\beta,\gamma,\delta}_\rr(x)$ is invariant
under the transpositions $\al\leftrightarrow\be$ and $\ga\leftrightarrow\de$.
This can be verified using transformation
formulas for the ${}_2\phi_1$ series, see Gasper--Rahman \cite[(III.3) and (II.2)]{GR}:
\begin{align*}
{}_2\phi_1\left(\begin{matrix} A, B\\
C\end{matrix}\bigg| Z\right)&= \frac{(ABZ/C;q)_\infty}{(Z;q)_\infty}  {}_2\phi_1\left(\begin{matrix} C/A, C/B\\
C\end{matrix}\bigg| ABZ/C\right)\\
&= \frac{(C/B,BZ;q)_\infty}{(C,Z;q)_\infty} {}_2\phi_1\left(\begin{matrix} ABZ/C, B\\
BZ\end{matrix}\bigg| C/B\right).
\end{align*}
\end{remark}
Further, for $\rr\in\mathbb Z$ we set
\begin{equation}\label{eq_F_norm}
  \mathfrak h_{\rr}= \z_+
  \frac{(\ga\de)^{\rr}}{\al\be}
\,
 \cdot
  \frac{q^{
 2-\rr^2} }{q^{3-2\rr} \frac{\ga
\de}{\al \be}-1} \cdot \dfrac{\theta_q\left(\dfrac{\z_-}{\z_+}, \; \ga\de
\z_-\z_+\right)}{\theta_q(\ga \z_-,\; \de  \z_-,\; \ga  \z_+,\; \de  \z_+)} \cdot
\dfrac{\left(q,\; q, \; \dfrac\al{\de}q^{\rr-1}, \; \dfrac\al{\ga }q^{\rr-1}, \;
\dfrac\be{\de}q^{\rr-1}, \; \dfrac\be{\ga}q^{\rr-1}
\right)_\infty}{\left(\dfrac{\al\be}{\ga\de}q^{2\rr-2},\;
\dfrac{\al\be}{\ga\de}q^{2\rr-2};\; q\right)_{\infty}}.
\end{equation}

\begin{theorem} \label{theorem_ergodic} Take any admissible quadruple $\alpha,\beta,\gamma,\delta$
such that $\alpha\beta<q^2\gamma\delta$. Then the correlation functions $\rho^n$,
$n=1,2,\dots$ of $\M_\infty^{\alpha,\beta,\gamma,\delta}$ are computed via
$$
 \rho^n(x_1,\dots,x_n)=\det [K^{\al,\be,\ga,\de}(x_i,x_j)]_{i,j=1}^n,
$$
$$
 K^{\al,\be,\ga,\de}(x,y)=
 \frac{1}{\mathfrak h_1} \cdot\frac{\F_0(x)\F_1(y)-\F_1(x)\F_0(y)}{x-y},
$$
where the singularity at $x=y$ in the last formula should be resolved using the
L'Hospital's rule.
\end{theorem}

\begin{remark}
 It might seem that the definition of $K^{\al,\be,\ga,\de}(x,y)$ through $\F_r$ and
 \eqref{eq_F_function} makes sense only when $\be<q\gamma$, since otherwise the
 series defining $_2\phi_1$ diverges. However,
$\alpha\beta<q^2\gamma\delta$ guarantees that either $|\al|<q|\gamma|$, or
$|\be|<q|\gamma|$, or $|\alpha|<q|\de|$, or $|\be|<q|\de|$. Since $\F_r$ does no
change when we
 swap $\alpha\leftrightarrow\beta$ and $\ga\leftrightarrow\de$ we can then extend the
 definition to all $4$ cases. In fact, we believe (but leave this out of the scope of the present paper)
  that Theorem \ref{theorem_ergodic}
 should hold for any admissible quadruple of parameters (that is, if instead of $\alpha\beta<q^2\gamma\delta$ we impose the weaker requirement $\alpha\beta<q\gamma\delta$ as in Definition \ref{def.A}) provided that we replace $K^{\al,\be,\ga,\de}(x,y)$
 by its analytic continuation in parameters $\al,\be,\ga,\de$.
\end{remark}

The rest of this section is devoted to the proof of Theorem \ref{theorem_ergodic}.
We start by computing correlation functions of the measures $\M_N^{\al,\be,\ga,\de}$.

\begin{proposition}
 \label{Proposition_corr_function} Take any admissible quadruple $\alpha,\beta,\gamma,\delta$.
 Let $w(x)=w(x;\al,\be,\ga q^{1-N},\de q^{1-N})$,
 $P_n(x)=P_n(x; \al,\be,\ga q^{1-N},\de q^{1-N})$ and let $h_n$, $k_n$ be the
 squared norms and leading coefficients of these polynomials, as defined in Section
 \ref{Section_orth_polynomials}.

 Then the correlation functions $\rho^n_N$, $n=1,2,\dots$ of
$\M_N^{\alpha,\beta,\gamma,\delta}$ are computed via
\begin{equation}
\label{eq_x13}
 \rho^n_N(x_1,\dots,x_n)=\det [K^{\al,\be,\ga,\de}_N(x_i,x_j)]_{i,j=1}^n,
\end{equation}
\begin{equation}
\label{eq_x14}
 K_N^{\al,\be,\ga,\de}(x,y)=
 \frac{k_{N-1} }
 {k_N\, h_{N-1}}
 \sqrt{w(x)w(y)}
  \, \frac{P_N(x)P_{N-1}(y)-P_{N-1}(x)P_N(y)}{x-y},
\end{equation} where the singularity at $x=y$ in the last formula should be resolved using the
L'Hospital's rule.
\end{proposition}

\begin{proof}
 The admissibility condition implies that the first $N$ orthogonal polynomials
 $P_0,\dots,P_{N-1}$ are well-defined and belong to $\ell_2(\L,w)$. Then a
 general theorem (see e.g. Borodin \cite{Borodin_det}, Borodin--Gorin \cite[Section 3]{BG_lectures}) for the $N$--point ensembles on the real line with probability
 distribution of the form
 $$
  {\rm Prob}(x_1,\dots,x_N)=\frac{1}{\Z} \prod_{i=1}^N w(x_i)  \prod_{1\le i<j\le
  N} (x_i-x_j)^2
 $$
 known as \emph{orthogonal polynomial ensemble} with weight $w(x)$ implies that the
 correlation functions are of the form \eqref{eq_x13} with
 $$
  K_N^{\al,\be,\ga,\de}(x,y)=\sqrt{w(x)w(y)} \cdot \sum_{n=0}^{N-1} \frac{P_n(x) P_n(y
  )}{h_n}.
 $$
Applying the Christoffel-Darboux formula (see e.g. Szeg\"o \cite[Section 3.2]{Szego}) we rewrite $K_N^{\al,\be,\ga,\de}(x,y)$
 as \eqref{eq_x14}. We do not have to worry if the polynomial $P_N(x)$ is square integrable, because the
Christoffel-Darboux formula relies only on the three-term relations $xP_n=\dots$
for $n=0,\dots,N-1$, which make sense.

In formula \eqref{eq_x14}, the indeterminacy on the diagonal $x=y$ is resolved via
the L'Hospital's rule. Alternatively, we may use the Cauchy integral and write
\begin{equation} \label{eq_Cauchy}
K_N(x,x)=\frac1{2\pi i}\,\oint K_N(x,y)\, \frac{dy}{y-x}
\end{equation}
with integration over a small simple $y$-contour around $x$.
\end{proof}

The correlation functions for $\M_{\infty}^{\al,\be,\ga,\de}$ can be obtained by the
limit transition of Corollary \ref{corrfunctions} and for that we need to send
$N\to\infty$ in all the parts of the formula for $K_N^{\al,\be,\ga,\de}(x,y)$. We
rely on the following statement.

\begin{lemma} For any fixed parameters $B,C,D,E$ such that $|D|>q$,
$D\notin q^{\mathbb Z}$, $E\notin q^{\mathbb Z}$, we have as integer $n$ tends to
$+\infty$
\begin{equation}\label{eq_phi32_conv}
{}_3\phi_2\left(\begin{matrix} q^{-n}, Bq^{-n}, Cq^{-n}\\
Dq^{-n}, Eq^{-n}\end{matrix}\bigg| q\right)\sim \left(\frac{BC}{DE}\right)^n\,
(-1)^n\, q^{-n(n-1)/2}\, \frac{\left(\frac{DE}{BC};q\right)_\infty}{(E^{-1}q;q)_\infty}
{}_2\phi_1\left(\begin{matrix} \frac DB, \, \frac DC\\ \frac
{DE}{BC}\end{matrix}\bigg| \frac qD\right).
\end{equation}
\end{lemma}

\begin{proof}
\eqref{eq_phi32_conv} is proved by applying the transformation formula in Gasper--Rahman
\cite[(III.11)]{GR}, which reads
\begin{equation}
\label{eq_phi32_transform}
{}_3\phi_2\left(\begin{matrix} q^{-n}, Bq^{-n}, Cq^{-n}\\
Dq^{-n}, Eq^{-n}\end{matrix}\bigg| q\right)=\dfrac{\left(\frac{DE}{BC};q\right)_n\,
\left(\frac{BC}D\right)^n
q^{-n^2}}{(Eq^{-n};q)_n}\, {}_3\phi_2\left(\begin{matrix} q^{-n}, \frac DB, \frac DC\\
Dq^{-n}, \frac{DE}{BC}\end{matrix}\bigg| q\right).
\end{equation}
Next, to handle $(Eq^{-n})_n$, we use
\begin{equation}\label{eq_pochamer_transform}
(Eq^{-n};q)_n=E^n (-1)^n q^{-n(n+1)/2}(E^{-1}q;q)_n.
\end{equation}
It remains to notice that as $n\to +\infty$, each term in the series expansion of
${}_3\phi_2$ in the right-hand side of \eqref{eq_phi32_transform} degenerates to the
same term in the expansion for ${}_2\phi_1$ indicated in \eqref{eq_phi32_conv}.
Since the series expansion of ${}_3\phi_2$ is majorated by a multiple of geometric
series with ratio $|q/D|<1$, this gives the desired convergence of ${}_3\phi_2$  to
${}_2\phi_1$ and proves \eqref{eq_phi32_conv}.
\end{proof}

As a corollary we find the asymptotics of the polynomials $P_n$.

\begin{proposition} \label{Prop_polynomial_limit} Take an admissible quadruple $\alpha,\beta,\gamma,\delta$
and any $\rr\in\mathbb Z$ such that $P_{N-\rr}(x)$ is a well--defined orthogonal
polynomial for all large $N$, and such that $|q^{\rr-1}\beta|<|\gamma|$. We have
\begin{multline} \label{eq_x15}
\lim_{N\to\infty} (\sgn x )^{N-1}
  \frac{(\gamma\delta)^{(N-1)/2}}{q^{N(N-1)/2}}\sqrt{w(x;\al,\be,\ga q^{1-N},\de q^{1-N})}\,
\frac{P_{N-\rr}(x;\al,\be,\ga q^{1-N},\de q^{1-N})}{k_{N-\rr}(\al,\be,\ga
q^{1-N},\de q^{1-N})}
\\=
  \F_\rr(x).
\end{multline}

The convergence in \eqref{eq_x15} holds for any $x\in\L$. Moreover, for each $x_0\in
\L$ the convergence in \eqref{eq_x15} is uniform over $x$ in an open complex
neighborhood of $x_0$ if we replace everywhere $\sgn(x)$ by $\sqrt{x^2}/x$ with
branch of square root chosen so that $\sqrt{x_0^2}/x_0=\sgn(x_0)$, and similarly
replace $|x|$ by $\sqrt{x^2}$ in all the definitions.
\end{proposition}

\begin{proof}
We start by investigating the asymptotic of $\sqrt{w(x)}$. We have
\begin{multline*}
w(x; \al,\be,\ga q^{1-N},\de q^{1-N})=w(x; \al,\be,\ga,\de)\\
\times\frac1{(1-\ga xq^{-1})\dots(1-\ga xq^{1-N})(1-\de xq^{-1})\dots(1-\de x q^{1-N})}\\
=w(x; \al,\be,\ga,\de)\frac{q^{N(N-1)}}{(\ga\de)^{N-1}}\,
\,\frac{x^{-2(N-1)}}{(\ga^{-1}x^{-1}q)_{N-1}(\de^{-1}x^{-1}q)_{N-1}}.
\end{multline*}

Note that $\ga\de>0$ while $x$ may be negative, so that we write
$$
\sqrt{x^{2(N-1)}}=(x\sgn x)^{N-1}.
$$
Therefore, as $N\to\infty$
$$
(\sgn x)^{N-1}\sqrt{w(x;\al,\be,\ga q^{1-N},\de q^{1-N})}\sim\frac{q^{N(N-1)/2}}{(\ga\de)^{(N-1)/2}x^{N-1}}\\
\sqrt{|x|\frac{(x\al,x\be;q)_\infty}{\theta_q(x\ga,x\de)}}.
$$

We further compute the large-$N$ asymptotics of $P_{N-\rr}(x;\al,\be,\ga q^{1-N},\de
q^{1-N}).$ By the definition of Section \ref{Section_orth_polynomials},
\begin{equation} \label{eq_x16}
P_{N-\rr}(x;\al,\be,\ga q^{1-N}, \de q^{1-N}) = {}_3\phi_2\left(\begin{matrix}
q^{-N+\rr}, \; \frac{\ga\de}{\al\be} q^{3-N-\rr},
 \; \ga x q^{1-N}\\
\dfrac\ga\be q^{2-N}, \; \dfrac\ga\al q^{2-N}\end{matrix}\Bigg| q\right).
\end{equation}
Now we use \eqref{eq_phi32_conv} with parameters
$$
n=N-\rr, \quad B=\frac{\ga\de}{\al\be}q^{3-2\rr}  , \quad C=\ga x q^{1-\rr}, \quad
D=\frac\ga\be q^{2-\rr}, \quad E=\frac\ga\al q^{2-\rr}.
$$

This gives
\begin{multline*}
P_{N-\rr}(x;\al,\be,\ga q^{1-N}, \de q^{1-N})\sim (-1)^{N-\rr}\cdot \de^{N-\rr}\cdot
x^{N-\rr}\cdot
 q^{-(N-\rr)(N-\rr-1)/2-\rr(N-\rr)} \\ \times
\frac{\left(\dfrac{q^{\rr}} {\delta x} \right)_\infty}{\left(\dfrac{\al}{\ga}
q^{\rr-1}\right)_\infty} {}_2\phi_1\left(\begin{matrix} \dfrac { q^{\rr-1}
{\al}}{{\de}}, \, \dfrac {q}{\be x }\\ \dfrac { q^{\rr} }{\de  x
}\end{matrix}\,\Bigg|\, \dfrac {q^{\rr-1}\beta}{\ga } \right)
\end{multline*}
For the leading coefficients $k_n$ we use the formula \eqref{eq_leading_term} and
transformation \eqref{eq_pochamer_transform} which yield
\begin{multline*}
 k_{N-\rr}(\al,\be,\ga
q^{1-N},\de q^{1-N})=q^{(1-N)(N-\rr)} \gamma^{N-\rr} \dfrac{(\frac{ \ga\de}{\al\be}
q^{-2\rr+3} q^{\rr-N};q)_{N-\rr}}{ (\frac{\ga}{\be} q^{2-\rr} q^{\rr-N},
\frac{\ga}{\al} q^{2-\rr} q^{\rr-N};q)_{N-\rr}}\\ = (-1)^{N-\rr}
q^{(N-\rr)(-N-\rr+1)/2}\dfrac{(\frac{\al\be}{ \ga\de} q^{2\rr-2};q)_{N-\rr}}{
(\frac{\be}{\ga} q^{\rr-1}, \frac{\al}{\ga} q^{\rr-1} ;q)_{N-\rr}} 
\de^{N-\rr}
\end{multline*}
Combining the asymptotic of $\sqrt{w(x)}$, $P_{N-\rr}(x)$, and $k_{N-\rr}$  we
arrive at the desired result.
\end{proof}

\begin{proposition} \label{Prop_limit_norm}
 As $N\to\infty$
 $$
  \frac{h_{N-\rr}(\alpha,\beta,\gamma q^{1-N},\delta q^{1-N})}
  {[k_{N-\rr}(\alpha,\beta,\gamma q^{1-N},\delta q^{1-N})]^2}\sim   \frac{q^{N(N-1)} }{
  (\ga\de)^{N-1}} \cdot \mathfrak h_\rr.
 $$
\end{proposition}

\begin{remark}
 In fact, under appropriate restrictions on the quadruple
 $\alpha,\beta,\gamma,\delta$ and $\rr$ the functions $\F_\rr$ are orthogonal in
 $\ell^2(\L)$ with counting measure and $\mathfrak h_\rr$ are their squared norms.
 This is merely the $N\to\infty$ limit of the orthogonality relations for polynomials
 $P_n$.
\end{remark}

\begin{proof}[Proof of Proposition \ref{Prop_limit_norm}]
By the formulas of Section \ref{Section_orth_polynomials}, we have
\begin{multline}
\label{eq_hn} \frac{h_{n}(\al,\be,\ga q^{1-N},\de
q^{1-N},\z_+,\z_-)}{[k_n(\al,\be,\ga q^{1-N},\de q^{1-N})]^2}=\z_+\dfrac{\left(q,\;
\dfrac{\al}{\ga}q^{N-1}, \; \dfrac {\al}{\de}q^{N-1}, \; \dfrac{\be}{\ga} q^{N-1},
\; \dfrac{\be}{\de} q^{N-1};\; q
\right)_\infty} {\left(q^{2N-3} \cdot\frac{\al \be}{\ga \de};\; q\right)_\infty}\\
\times \dfrac{\theta_q\left(\dfrac{\z_-}{\z_+}, \; \ga\de q^{2-2N}
\z_-\z_+\right)}{\theta_q(\ga q^{1-N} \z_-,\;
\de q^{1-N} \z_-,\; \ga q^{1-N} \z_+,\; \de q^{1-N} \z_+)}\\
\times (-1)^n\cdot\frac{q^{2n+n(n-1)/2}}{(\al \be)^n}\cdot \frac{q^{3-2N} \frac{\ga
\de}{\al \be}-1}{q^{3-2N+2n} \frac{\ga \de}{\al \be}-1}\cdot \dfrac{\left(\dfrac{\ga
q^{2-N}}\al, \; \dfrac{\ga q^{2-N}}\be, \; \dfrac{\de q^{2-N}}\al,\; \dfrac{\de
q^{2-N}}\be, \; q; q\right)_n}{\left(q^{3-2N+n} \frac{\ga \de}{\al \be},\;
q^{3-2N+n} \frac{\ga \de}{\al \be},\; q^{3-2N} \frac{\ga \de}{\al \be};\;
q\right)_n}.
\end{multline}
Set $n=N-\rr$ and send $N\to\infty$. Then using \eqref{eq_pochamer_transform} and
the identity
$$
 \theta_q(qx)=(qx;q)_\infty (x^{-1};q)_\infty = (x;q)_\infty (q x^{-1};q)_\infty
 \frac{1-x^{-1}}{1-x}=-x^{-1} \theta_q(x),
$$
we see that asymptotically \eqref{eq_hn} behaves as
\begin{multline}
\label{eq_hn2} \z_+ \cdot (q;q)^2_\infty  q^{-(N-1)(2N-1)+2N(N-1)}
\dfrac{\theta_q\left(\dfrac{\z_-}{\z_+}, \; \ga\de  \z_-\z_+\right)}{\theta_q(\ga
\z_-,\; \de  \z_-,\; \ga  \z_+,\; \de  \z_+)} \cdot
 q^{2n+n(n-1)/2} \frac{q^{-2N}}{
(\ga\de)^n} \cdot \frac{q^{3} \frac{\ga \de}{\al \be}}{q^{3-2N+2n} \frac{\ga
\de}{\al \be}-1}
\\
\times
 \dfrac{\left(\dfrac\al{\ga q^{1+n-N}}, \; \dfrac\be{\ga q^{1+n-N}}, \; \dfrac\al{\de
q^{1+n-N}},\; \dfrac\be{\de q^{1+n-N}};\; q \right)_n}{\left(q^{-2+2N-2n} \frac{\al
\be}{\ga \de},\; q^{-2+2N-2n} \frac{\al \be}{\ga \de},\; q^{-2+2N-n} \frac{\al
\be}{\ga \de};\; q\right)_n } q^{\frac{n}{2} (4(3-2N+n)-2(5-4N+3n)-(5-4N+n))}
\\
\sim
 q^{N(N-1)}  \frac{1}{ (\ga\de)^{N-\rr}}
\cdot q^{
 2-\rr^2}\,
 \z_+ \cdot (q;q)^2_\infty
\dfrac{\theta_q\left(\dfrac{\z_-}{\z_+}, \; \ga\de  \z_-\z_+\right)}{\theta_q(\ga
\z_-,\; \de  \z_-,\; \ga  \z_+,\; \de  \z_+)} \cdot
 \frac{ \frac{\ga \de}{\al \be}}{q^{3-2\rr} \frac{\ga
\de}{\al \be}-1}
\\
\times
 \dfrac{\left(\dfrac\al{\ga q^{1-\rr}}, \; \dfrac\be{\ga q^{1-\rr}}, \; \dfrac\al{\de
q^{1-\rr}},\; \dfrac\be{\de q^{1-\rr}};\; q \right)_\infty}{\left(q^{-2+2\rr}
\frac{\al \be}{\ga \de},\; q^{-2+2\rr} \frac{\al \be}{\ga \de};\; q\right)_\infty }.
\qedhere
\end{multline}
\end{proof}

\begin{proof}[Proof of Theorem \ref{theorem_ergodic}]
 By Proposition \ref{Proposition_corr_function} and Corollary
 \ref{corrfunctions}
 we need to study $N\to\infty$ asymptotic of $K_N^{\al,\be,\de,\ga}(x,y)$. Note that
 both coherent systems $\M_N^{\al,\be,\ga,\de}$ and the kernel
 $K_N^{\al,\be,\de,\ga}(x,y)$ is invariant under swaps $\alpha \leftrightarrow
 \beta$ and $\ga \leftrightarrow \de$. Therefore, we can assume by making an
 appropriate swap that $|\beta|< q|\gamma|$, since $\alpha\beta< q^2 \ga\de$
 guarantees that either $|\beta|< q|\gamma|$ or $|\al|< q|\de|$.

 We now consider the case $x\ne y$. Then we use Propositions \ref{Prop_polynomial_limit} and
 \ref{Prop_limit_norm} to get the required asymptotics and notice that the factors
 $(\sgn x)^{N-1} (\sgn y)^{N-1}$ cancel out when we compute the determinants in
 the definition of $\rho_n^N$.

 We extend the asymptotic to the case $x=y$ by using the Cauchy integral (since before the limit we
 deal with polynomials, everything is analytic) and then passing in it to
 $N\to\infty$ limit by using the uniformity in Proposition
 \ref{Prop_polynomial_limit}.
\end{proof}

\bigskip

\noindent
Vadim Gorin\\
Department of Mathematics, MIT, Cambridge, MA, USA; \\
Institute for Information Transmission Problems, Moscow, Russia\\
\texttt{vadicgor@gmail.com}

\medskip
\noindent
Grigori Olshanski\\
Institute for Information Transmission Problems, Moscow, Russia;\\
National Research University Higher School of Economics, Moscow, Russia\\
\texttt{olsh2007@gmail.com}

\end{document}